\documentclass[12pt,a4paper,reqno]{amsart}
\usepackage[utf8]{inputenc}
\usepackage[T1]{fontenc}
\usepackage{indentfirst}
\usepackage{textcomp,lscape,rotating}
\usepackage{mathpazo}
\usepackage[english]{babel}
\usepackage{enumerate}
\usepackage{amsmath,amsfonts,amssymb,amsopn,amscd,amsthm}
\usepackage{bbm}
\usepackage{stmaryrd}
\usepackage{graphicx}
\usepackage[dvipsnames]{xcolor}
\usepackage[colorlinks=true,citecolor=DarkOrchid,linkcolor=black]{hyperref}

\newcommand{\N}{\mathbb{N}}
\newcommand{\Z}{\mathbb{Z}}
\newcommand{\R}{\mathbb{R}}
\newcommand{\CC}{\mathbb{C}}
\newcommand{\E}{\mathrm{e}}
\newcommand{\I}{\mathrm{i}}
\newcommand{\esper}{\mathbb{E}}
\newcommand{\proba}{\mathbb{P}}
\newcommand{\qproba}{\mathbb{Q}}

\newcommand{\cym}{\mathfrak{C}}
\newcommand{\pym}{\mathfrak{P}}
\newcommand{\qym}{\mathfrak{Q}}
\newcommand{\tym}{\mathfrak{T}}
\newcommand{\nym}{\mathfrak{N}}
\newcommand{\dym}{\mathfrak{D}}

\newcommand{\CW}{\mathbb{C}\mathbb{W}}
\newcommand{\IS}{\mathbb{I}}
\newcommand{\card}{\mathrm{card}}

\newcommand{\eps}{\varepsilon}
\newcommand{\leb}{\mathrm{L}}
\newcommand{\lle}{\left[\!\left[} 
\newcommand{\rre}{\right]\!\right]} 

\newcommand{\comment}[1]{}
\newcommand{\figcap}[2]{
\begin{center}
\begin{figure}[ht]
\includegraphics{#1}
\caption{#2}
\end{figure}
\end{center}\vspace{-0.5cm}
}

\newtheorem{theorem}{Theorem}
\newtheorem{proposition}[theorem]{Proposition}
\newtheorem{lemma}[theorem]{Lemma}
\newtheorem{definition}[theorem]{Definition}
\newtheorem{corollary}[theorem]{Corollary}

\theoremstyle{remark}
\newtheorem*{remark}{Remark}
\newtheorem*{example}{Example}

\title[Mod-Gaussian convergence and applications for models of statistical mechanics]{Mod-Gaussian convergence and its applications\\ for models of statistical mechanics}
\author{Pierre-Lo\"ic M\'eliot and Ashkan Nikeghbali}
\date{\today}

\begin{document}

\begin{abstract}
In this paper we complete our understanding of the role played by the limiting (or residue) function in the context of mod-Gaussian convergence. The question about the probabilistic interpretation of such functions was initially raised by Marc Yor. After recalling our recent result which interprets the limiting function as a measure of "breaking of symmetry" in the Gaussian approximation in the framework of general central limit theorems type results, we introduce the framework of $\mathrm{L}^1$-mod-Gaussian convergence in which the residue function is obtained as (up to a normalizing factor) the probability density of some sequences of random variables converging in law after a change of probability measure. In particular we recover some celebrated results due to Ellis and Newman on the convergence in law of dependent random variables arisisng in statistical mechanics. We complete our results by giving an alternative approach to the Stein method to obtain the rate of convergence in the Ellis-Newman convergence theorem and by proving a new local limit theorem. More generally we illustrate our results with simple models from statistical mechanics.
\end{abstract}

\maketitle

\begin{center}
\dedicatory{\emph{In memoriam, Marc Yor}.}
\end{center}
\bigskip

\section{Introduction}

Let $(X_n)_{n \in \N}$ be a sequence of real-valued random variables. In the series of papers \cite{JKN11,DKN11,KN10,KN12,FMN13}, we introduced the notion of mod-Gaussian convergence (and more generally of mod-convergence with respect to an infinitely divisible law $\phi$):

\begin{definition}\label{def:mod}
The sequence $(X_n)_{n \in \N}$ is said to converge in the mod-Gaussian sense with parameters $t_n \to +\infty$ and limiting (or residue) function $\theta$ if, locally uniformly in $\R$,
$$\esper[\E^{\I t X_n}]\,\E^{\frac{t_n\,t^2}{2}} = \theta(t) \,(1+o(1)),$$
where $\theta$ is a continuous function on $\R$ with $\theta(0)=1$. 

\end{definition}

\noindent A trivial situation of mod-Gaussian convergence is when $X_n=G_n+Y_n$ is the sum of a Gaussian variable of variance $t_n$ and of an independent random variable $Y_n$ that converges in law to a variable $Y$ with characteristic function $\theta$. More generally $X_n$ can be thought of as a Gaussian variable of variance $t_n$, plus a noise which is encoded by the multiplicative residue $\theta$ in the characteristic function. In this setting, $\theta$ is not necessarily the characteristic function of a random variable (the residual noise). For instance, consider $$X_n=\frac{1}{n^{1/3}}\sum_{i=1}^n Y_i,$$ where the $Y_i$ are centred, independent and identically distributed random variables with convergent moment generating function. Then a Taylor expansion of $\esper[\E^{\I t Y}]$ shows that $(X_n)_{n \in \N}$ converges in the mod-Gaussian sense with parameters $n^{1/3}\,\mathrm{Var}(Y)$ and limiting function $$\theta(t)=\exp\left(\frac{\esper[Y^3]\,(\I t)^3}{6}\right),$$
which is not the characteristic function of a random variable, since it does not go to zero as $t$ goes to infinity. In 2008, during the workshop "\textit{Random matrices, L-functions and primes}" held in Z\"urich, Marc Yor asked the second author A.~N. about the role of the limiting function $\theta$. In \cite{KNN13} it is proved that the set of possible limiting functions is the set of continuous functions $\theta$ from $\R$ to $\CC$ such that $\theta(0)=1$ and $\theta(-t)=\bar\theta(t)$ for $t\in\R$. But this characterization does not say anything on the probabilistic information encoded in $\theta$.  We now wish to develop more on  probabilistic interpretations of the limiting function and the implications of mod-Gaussian convergence in terms of classical limit theorems of probability theory.
\bigskip

We first note that  by looking at $\esper[\E^{\I t X_n/\sqrt{t_n}}]$, one immediately sees that mod-Gaussian convergence implies a central limit theorem for the sequence $(\frac{X_n}{\sqrt{t_n}})$:
\begin{equation}
\frac{X_n}{\sqrt{t_n}} \rightharpoonup_{n \to \infty} \mathcal{N}(0,1),\label{eq:clt}
\end{equation}
where the convergence above holds in law (see \cite[\S 2-3]{JKN11} for more details on this). On the other hand, with somewhat stronger hypotheses on the remainder $o(1)$ that appears in Definition \ref{def:mod}, a local limit theorem also holds, see \cite[Theorem 4]{KN12} and \cite[Theorem 5]{DKN11}:
$$\proba[X_n \in B] = \proba[G_n \in B]\, (1+o(1)) = \frac{m(B)}{\sqrt{2\pi t_n}}\,(1+o(1))$$
for relatively compact sets $B$ with $m(\partial B)=0$, $m$ denoting the Lebesgue measure. \bigskip

In \cite{FMN13}, it is then explained that by looking at Laplace transforms instead of characteristic functions, and by assuming the convergence holds on a whole band of the complex plane, one can obtain in the setting of mod-Gaussian convergence precise estimates of moderate or large deviations. In fact these results provide a new probabilistic interpretation of the limiting function as a measure of the "breaking of symmetry"  in the Gaussian approximation of the tails of $X_n$ (see \S\ref{subsec:brisesymm} for more details).
\bigskip

The goal of this paper is threefold:\vspace{2mm}
\begin{itemize}
\item to propose a new interpretation of the limiting function in the framework of mod-Gaussian convergence with Laplace transforms; these results allow us in particular to recover some well known exotic limit theorems from statistical mechanics due to Ellis and Newman \cite{EN78} and similar one for other models or in higher dimensions.\vspace{2mm}
\item to show that once one is able to prove mod-Gaussian convergence, then one can expect to obtain finer results than merely convergence in law, such as speed of convergence and local limit theorems. Results on the rate of convergence in the Curie-Weiss model were recently obtained using Stein's method (see \emph{e.g.} \cite{EL10}) while the local limit theorem, to the best of our knowledge, is new.\vspace{2mm}
\item to explore the applications of the results obtained in \cite{FMN13} on the "breaking of symmetry" in the central limit theorem  to some classical models of statistical mechanics. In particular our approach determines the scale up to which the Gaussian approximation for the tails is valid and its breaking at this critical scale.\vspace{2mm}
\end{itemize}
Our results are best illustrated with some classical one-dimensional models from statistical mechanics, such as the Curie-Weiss model or the Ising model. To illustrate the flexibility of our approach, we shall also prove similar results for weighted symmetric random walks in dimensions $2$ and $3$. The statistics of interest to us will be the total magnetization, which can be written as a sum of dependent random variables. These examples add to the already large class of examples of sums of dependent random variables we have already been able to deal with in the context of mod-$\phi$ convergence. 
\bigskip

In the remaining of the introduction we recall the results obtained in \cite{FMN13} which led us to the "breaking of symmetry" interpretation, as well as an underlying method of cumulants that enabled us to establish the mod-Gaussian convergence for a large family of sums of dependent random variables. The important aspect of the cumulant method is that it provides a tool to prove mod-Gaussian convergence in situations where one cannot explicitly compute the characteristic function. We eventually give an outline of the paper.
\medskip

\subsection{Complex convergence and interpretation of the residue}\label{subsec:brisesymm}
We consider again a sequence of real-valued random variables $(X_n)_{n \in \N}$, but this time we assume that their Laplace transforms $\esper[\E^{zX_n}]$ are convergent in an open disk of radius $c>0$. In this case, they are automatically well-defined and holomorphic in a band of the complex plane $\mathcal{B}_c = \{z \in \CC,\,\,|\mathrm{Re}(z)|<c\}$ (see \cite[Theorem 6]{LS52}, and \cite{Ess45} for a general survey of the properties of Laplace and Fourier transforms of probability measures).\medskip

\begin{definition}\label{def:mod2}
The sequence $(X_n)_{n \in \N}$ is said to converge in the complex mod-Gaussian sense with parameters $t_n$ and limiting function $\psi$ if, locally uniformly on $\mathcal{B}_c$,
$$\esper[\E^{z X_n}]\,\E^{-\frac{t_n\,z^2}{2}} = \psi(z) \,(1+o(1)),$$
where $\psi$ is a continuous function on $\mathcal{B}_c$ with $\psi(0)=1$. Then, one has in particular convergence in the sense of Definition \ref{def:mod}, with $\theta(t)=\psi(\I t)$.
\end{definition}

In this setting which is more restrictive than before, the residue $\psi$ has a natural interpretation as a measure of "breaking of symmetry" when one tries to push the estimates of the central limit theorem from the scale $\sqrt{t_n}$ to the scale $t_n$. The previously mentioned central limit theorem \eqref{eq:clt} tells us that:
$$\proba\!\left[ X_n \geq a\sqrt{t_n}\right] = \left(\frac{1}{\sqrt{2\pi}} \int_{a}^\infty \E^{-\frac{x^2}{2}}\,dx \right)(1+o(1))$$
for any $a \in \R$. In the setting of complex mod-Gaussian convergence, this estimate remains true with $a=o(\sqrt{t_n})$, so that if $\eps=o(1)$, then
\begin{align*}\proba\left[ X_n \geq \eps\, t_n\right] &=\left(\frac{1}{\sqrt{2\pi}} \int_{\eps \sqrt{t_n}}^\infty \E^{-\frac{x^2}{2}}\,dx \right)(1+o(1)),\\
&=\frac{\E^{-\frac{t_n \eps^2}{2}}}{\sqrt{2\pi t_n} \,\eps}\,(1+o(1))\quad\text{ if }1 \gg \eps \gg \frac{1}{\sqrt{t_n}}\,,
\end{align*}where the notation $a_n \gg b_n$ stands for $b_n=o(a_n)$.
Then, at scale $t_n$, the limiting residue $\psi$ comes into play, with the following estimate that holds without additional hypotheses than those in Definition \ref{def:mod2}:
\begin{equation}\forall x \in (0,c),\,\,\,\proba[X_n \geq xt_n] = \frac{\E^{-\frac{t_n x^2}{2}}}{\sqrt{2\pi t_n} \,x}\,\psi(x)\,(1+o(1)),
\label{eq:upbound}
\end{equation}
the remainder $o(1)$ being uniform when $x$ stays in a compact set of $\R_+^*\cap(0,c)$. This estimate of positive large deviations has the following counterpart on the negative side:
$$\forall x  \in (0,c),\,\,\,\proba[X_n \leq -xt_n] = \frac{\E^{-\frac{t_n x^2}{2}}}{\sqrt{2\pi t_n} \,x}\,\psi(-x)\,(1+o(1)).$$
So for instance, if $(Y_n)_{n \in \N}$ is a sequence of i.i.d.~random variables with convergent moment generating function, mean $0$, variance $1$ and third moment $\esper[Y]>0$, then $X_n=\frac{1}{n^{1/3}}\sum_{i=1}^n Y_i$ converges in the \emph{complex} mod-Gaussian sense with parameters $n^{1/3}$ and limiting function $\psi(z)=\exp(\esper[Y^3]\,z^3/6)$, and therefore for $x>0$,
$$\proba\!\left[\sum_{i=1}^n Y_i \geq xn^{2/3}\right] = \proba\!\left[\mathcal{N}(0,1)\geq xn^{1/6}\right]\,\,\exp\!\left(\frac{\esper[Y^3]\,x^3}{6}\right)\,(1+o(1)).$$
Thus, at scale $n^{2/3}$, the fluctuations of the sum of i.i.d.~random variables are no more Gaussian, and the residue $\psi(x)$ measures this "breaking of symmetry": in the previous example, it makes moderate deviations on the positive side more likely than moderate deviations on the negative side, since $\psi(x)>1>\psi(-x)$ for $x>0$. \bigskip

\begin{remark}
The problem of finding the normality zone, i.e. the scale up to which the central limit theorem is valid, is a known problem in the case of i.i.d. random variables (see \emph{e.g.} \cite{IL71}). The description of the "symmetry breaking" is new and moreover the mod-Gaussian framework covers many examples with dependent  random variables (see also \cite{FMN13} for more examples). 
\end{remark}

Thus, the observation of large deviations of the random variables $X_n$ provides a first probabilistic interpretation of the residue $\psi$ in the deconvolution of a sequence of characteristic functions of random variables by a sequence of large Gaussian variables. In Section \ref{sec:l1}, we shall provide another interpretation of $\psi$, which is inspired by some classical results from statistical mechanics (\emph{cf.} \cite{EN78,ENR80}).
\medskip

\subsection{The method of joint cumulants}\label{subsec:methjoint}
The appearance of an exponential of a monomial $Kx^{r \geq 3}$ as the limiting residue in mod-Gaussian convergence is a phenomenon that occurs not only for sums of i.i.d.~random variables, but more generally for sums of possibly non identically distributed and/or dependent random variables. For instance, \vspace{1mm}
\begin{enumerate}
\item the number of zeroes of a random Gaussian analytic function $\sum_{k=0}^\infty (\mathcal{N}_{\CC})_k\,z^k$ in the disk of radius $1-\frac{1}{n}$;\vspace{1mm}
\item the number of triangles in a random Erd\"os-R\'enyi graph $G(n,p)$;\vspace{1mm}
\end{enumerate}
are both mod-Gaussian convergent after proper rescaling, and with limiting function of the form $\exp(Lz^3)$, with the constant $L$ depending on the model (see again \cite{FMN13}). The reason behind these universal asymptotics lies in the following method of cumulants. If $X$ is a random variable with convergent Laplace transform $\esper[\E^{zX}]$ on a disk, we recall that its cumulant generating function is 
\begin{equation}
\log \esper[\E^{zX}]=\sum_{r \geq 1}\frac{\kappa^{(r)}(X)}{r!}\,z^r,\label{eq:cumulant}
\end{equation}
which is also well-defined and holomorphic on a disk around the origin. Its coefficients $\kappa^{(r)}(X)$ are the cumulants of the variable $X$, and they are homogenenous polynomials in the moments of $X$; for instance, $\kappa^{(1)}(X)=\esper[X]$, $\kappa^{(2)}(X)=\esper[X^2]-\esper[X]^2$, and $\kappa^{(3)}(X)=\esper[X^3]-3\,\esper[X^2]\,\esper[X]+2\,\esper[X]^3$. \bigskip

Consider now a sequence of random variables $(W_n)_{n \in \N}$ with $\kappa^{(1)}(W_n)=0$, and for $r \geq 2$, 
\begin{equation}\kappa^{(r)}(W_n) = K_r \,\alpha_n\,(1+o(1)),\label{asymptoticscumulant}
\end{equation}
with $\alpha_n \to +\infty$. This assumption is inspired by the case of a sum $W_n=\sum_{i=1}^n Y_i$ of centred i.i.d.~random variables for which $\kappa^{r}(W_n)=n\,\kappa^{(r)}(Y)$. If it is satisfied, then  one can formally write
\begin{align*}
\log &\,\esper\!\left[\E^{z\, \frac{W_n}{(\alpha_n)^{1/3}}}\right]\\
&= (\alpha_n)^{-2/3}\,\frac{\kappa^{(2)}(W_n)\,z^2}{2}+(\alpha_n)^{-1}\,\frac{\kappa^{(3)}(W_n)\,z^3}{6}+\sum_{r \geq 4}\frac{\kappa^{(r)}(W_n)}{r!} ((\alpha_n)^{-1/3} z)^r\\
&\simeq (\alpha_n)^{1/3}\,\frac{K_2\,z^2}{2}+\frac{K_3\,z^3}{6}+\sum_{r \geq 4}\frac{K_r\,z^r}{r!} (\alpha_n)^{1-r/3}\\
&\simeq (\alpha_n)^{1/3}\,\frac{K_2\,z^2}{2}+\frac{K_3\,z^3}{6}
\end{align*}
whence the mod-Gaussian convergence of $X_n=(\alpha_n)^{-1/3}\,W_n$, with parameters $t_n=K_2\,(\alpha_n)^{1/3}$ and limiting function $\exp(K_3 \,z^3/6)$. The approximation is valid if the $o(1)$ in the asymptotics of $\kappa^{(2)}(W_n)$ is small enough (namely $o((\alpha_n)^{-1/3})$), and if the series $\sum_{r \geq 4}$ can be controlled, which is the case if
\begin{equation}\forall r,\,\, |\kappa^{(r)}(W_n)|\leq (Cr)^r\,\alpha_n\label{eq:boundcumulant}
\end{equation}
for some constant $C$. The method of cumulants in the setting of mod-Gaussian convergence amounts to prove \eqref{asymptoticscumulant} for the first cumulants of the sequence $(X_n)_{n \in \N}$, and \eqref{eq:boundcumulant} for all the other cumulants. From such estimates one then obtains mod-Gaussian convergence for an appropriate renormalisation of $(W_n)_{r \geq 3}$, with limiting function $\exp(K_r\,z^r/r!)$, where $r$ is the smallest integer greater or equal than $3$ such that $K_r \neq 0$.\bigskip

This method of cumulants works well with sequences $(W_n)_{n \in \N}$ that write as sums of (weakly) dependent random variables. Indeed, cumulants admit the following generalization to families of random variables, see \cite{LS59}. Denote $\qym_r$ the set of partitions of $\lle 1,r\rre = \{1,2,3,\ldots,r\}$, and $\mu$ the M\"obius function of this poset (see \cite{Rot64} for basic facts about M\"obius functions of posets): 
$$\mu(\Pi)=(-1)^{\ell(\Pi)-1}\,(\ell(\Pi)-1)!$$
where $\ell(\Pi)=s$ if $\Pi=\pi_1 \sqcup \pi_2 \sqcup \cdots \sqcup \pi_s$ has $s$ parts. The joint cumulant of a family of $r$ random variables with well defined moments of all order is
$$\kappa(X_1,\ldots,X_r)=\sum_{\Pi \in \qym_r} \mu(\Pi)\,\prod_{i=1}^{\ell(\Pi)}\esper\!\left[\prod_{j \in \pi_i} X_j\right].$$
It is multilinear and generalizes Equation \eqref{eq:cumulant}, since
\begin{align*}
\kappa(X_1,\ldots,X_r)&=\left.\frac{\partial^r}{\partial z_1 \partial z_2 \cdots \partial z_r}\right|_{z_1=\cdots=z_r=0}\big(\log \esper[\E^{z_1X_1+\cdots+z_rX_r}]\big)\\
\kappa\big(\underbrace{X,\ldots,X}_{r\text{ times}}\big)&=\kappa^{(r)}(X).
\end{align*}
Suppose now that $W_n=W=\sum_{i=1}^n Y_i$ is a sum of dependent random variables. By multilinearity,
\begin{equation}\kappa^{(r)}(W)=\sum_{i_1,\ldots,i_r} \kappa(Y_{i_1},\ldots,Y_{i_r}),\label{eq:multilinearity}
\end{equation}
so in order to obtain the bound \eqref{eq:boundcumulant}, it suffices to bound each "elementary" joint cumulant $\kappa(Y_{i_1},\ldots,Y_{i_r})$. To this purpose, it is convenient to introduce the dependency graph of the family of random variables $(Y_{1},\ldots,Y_{n})$, which is the smallest subgraph $G$ of the complete graph on $n$ vertices such that the following property holds: if $(Y_i)_{i \in I}$ and $(Y_j)_{j \in J}$ are disjoint subsets of random variables with no edge of $G$ between a variable $Y_i$ and a variable $Y_j$, then $(Y_i)_{i \in I}$ and $(Y_j)_{j \in J}$ are independent. Then, in many situations, one can write a bound on the elementary cumulant $\kappa(Y_{i_1},\ldots,Y_{i_r})$ that only depends on the induced subgraph $G[i_1,\ldots,i_r]$ obtained from the dependency graph by keeping only the vertices $i_1,\ldots,i_r$ and the edges between them. In particular:\vspace{2mm}
\begin{enumerate}
\item $\kappa(Y_{i_1},\ldots,Y_{i_r})=0$ if the induced graph $G[i_1,\ldots,i_r]$  is not connected.
 \vspace{2mm}
\item if $|Y_i| \leq 1$ for all $i$, then $|\kappa(Y_{i_1},\ldots,Y_{i_r})| \leq 2^{r-1}\,\mathrm{ST}(G[i_1,\ldots,i_r])$, where $\mathrm{ST}(H)$ is the number of spanning trees on a (connected) graph $H$. \vspace{1mm}
\end{enumerate}
By gathering the contributions to the sum of Formula \eqref{eq:multilinearity} according to the nature and position of the induced subgraph $G[i_1,\ldots,i_r]$ in $G$, one is able to prove efficient bounds on cumulants of sums of dependent variables, and to apply the method of cumulants to get their mod-Gaussian convergence. We refer to \cite{FMN13} for precise statements, in particular in the case where each vertex in $G$ has less than $D$ neighbors, with $D$ independent of the vertex and of $n$. In Section \ref{sec:cumulantspin}, we shall apply this method to a case where $G$ is the complete graph on $n$ vertices, but where one can still find correct bounds (and in fact exact formulas) for the joint cumulants $\kappa(Y_{i_1},\ldots,Y_{i_r})$: the one-dimensional Ising model.
\medskip

\subsection{Basic models}
 As mentioned above, the goal of the paper is to study the phenomenon of mod-Gaussian convergence for probabilistic models stemming from statistical mechanics; this extends the already long list of models for which we were able to establish this asymptotic behavior of the Fourier or Laplace transforms (\cite{JKN11,KN12,FMN13}). More precisely, we shall focus on one-dimensional spin configurations, which already yield an interesting illustration of the theory and technics of mod-Gaussian convergence. Given two parameters $\alpha \in \R$ and $\beta \in \R_+$, we recall that the Curie-Weiss model and the one-dimensional Ising model are the probability laws on spin configurations $\sigma : \lle 1,n\rre \to \{\pm 1\}$ given by
\begin{align}
\CW_{\alpha,\beta}(\sigma) &= \frac{1}{Z_n(\CW,\alpha,\beta)}\,\exp\left(\alpha \sum_{i=1}^n \sigma(i)+\frac{\beta}{2n} \left(\sum_{i=1}^n \sigma(i)\right)^{\!\!2\,}\right);\label{eq:curiemeasure}\\
\IS_{\alpha,\beta}(\sigma) &= \frac{1}{Z_n(\IS,\alpha,\beta)}\,\exp\left(\alpha \sum_{i=1}^n \sigma(i)+\beta \left(\sum_{i=1}^{n-1} \sigma(i)\,\sigma(i+1)\right)\right).\label{eq:isingmeasure}
\end{align}
The coefficient $\alpha$ measures the strength and direction of the exterior magnetic field, whereas $\beta$ measures the strength of the interaction between spins, which tend to align in the same direction. This interaction is local for the Ising model, and global for the Curie-Weiss model. Set $M_n=\sum_{i=1}^n \sigma(i)$: this is the total magnetization of the system, and a random variable under the probabilities $\CW_{\alpha,\beta}$ and $\IS_{\alpha,\beta}$.\bigskip

In Section \ref{sec:ising}, we quickly establish  the mod-Gaussian convergence of the magnetization for the Ising model, using the explicit form of the Laplace transform of the magnetization, which is given by the transfer matrix method. Alternatively, when $\alpha=0$, in the appendix, we apply the cumulant method and give an explicit formula for each elementary cumulant of spins (see Section \ref{sec:cumulantspin}). This allows us to prove the analogue for joint cumulants of the well-known fact that covariances between spins decrease exponentially with distance in the $1\mathrm{D}$-Ising model. This second method  is much less direct than the transfer matrix method, but we consider the Ising model to be a very good illustration of the method of joint cumulants. Moreover it illustrates the fact that one does not necessarily need to be able to compute precisely the moment generating function of the random variables.  
\bigskip

In Section \ref{sec:l1}, we focus on the Curie-Weiss model, and we interpret the magnetization as a change of measure on a sum of i.i.d.~random variables. Since these sums converge in the mod-Gaussian sense, it leads us to study the effect of a change of measure on a mod-Gaussian convergent sequence. We prove that in the setting of $\mathrm{L}^1$-mod-Gaussian convergence, such changes of measures either conserve the mod-Gaussian convergence (with different parameters), or lead to a convergence in law, with a limiting distribution that involves the residue $\psi$. We thus recover the results of \cite{EN78,ENR80}, and extend them to the setting of $\mathrm{L}^1$-mod-Gaussian convergence. In Section \ref{sec:rate}, using Fourier analytic arguments, we quickly recover the optimal rate of convergence of the Ellis-Newman limit theorem for the Curie-Weiss model which was recently obtained in \cite{EL10} using Stein's method, and then we establish a local limit theorem, thus completing the existing limit theorems for the Curie-Weiss model at critical temperature $\CW_{0,1}$. 
\bigskip
\bigskip

\section{Mod-Gaussian convergence for the Ising model:\\ the transfer matrix method}\label{sec:ising}

In this section, $(\sigma(i))_{i \in \lle 1,n\rre}$ is a random configuration of spins under the Ising measure \eqref{eq:isingmeasure}, and $M_n=\sum_{i=1}^n \sigma(i)$ is its magnetization. The mod-Gaussian convergence of $M_n$ after appropriate rescaling can be obtained by two different methods: the \emph{transfer matrix method}, which yields an explicit formula for $\esper[\E^{zM_n}]$; and the \emph{cumulant method}, which gives an explicit combinatorial formula for the coefficients of the power series $\log \esper[\E^{zM_n}]$. We use here the transfer matrix method, and refer to the appendix (Section \ref{sec:cumulantspin}) for the cumulant method.

\figcap{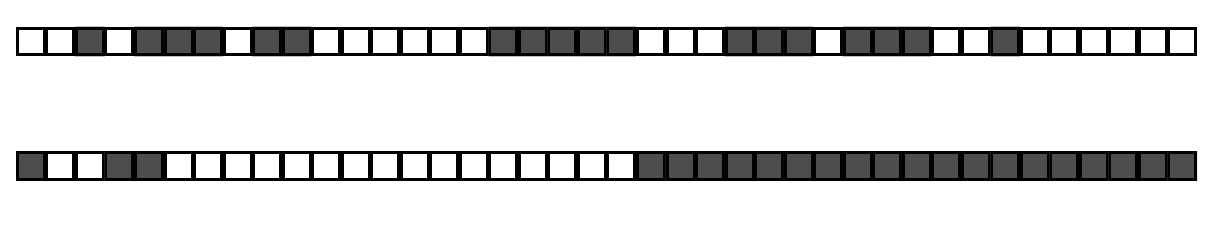}{Two configurations of spins under the Ising measures of parameters $(\alpha=0,\beta=0.3)$ and $(\alpha=0,\beta=1)$.}

 The Laplace transform $\esper[\E^{zM_n}]$ of the magnetization of the one-dimensional Ising model is well-known to be computable by the following transfert matrix method, see \cite[Chapter 2]{Bax82}. Introduce the matrix
$$T=\begin{pmatrix}
\E^{\alpha+\beta} & \E^{-\alpha-\beta} \\
\E^{\alpha-\beta} & \E^{-\alpha+\beta}
\end{pmatrix}, $$
and the two vectors $V=(\E^{\alpha},\E^{-\alpha})$ and $W=(\begin{smallmatrix}1 \\ 1\end{smallmatrix})$. If the rows and columns of $T$ correspond to the two signs $+1$ and $-1$, then any configuration of spins $\sigma=(\sigma(i))_{i \in \lle 1,n\rre}$ has under the Ising measure $\IS_{\alpha,\beta}$ a probability proportional to 
$$V_{\sigma(1)}T_{\sigma(1),\sigma(2)}T_{\sigma(2),\sigma(3)}\cdots T_{\sigma(n-1),\sigma(n)}.$$
Therefore, the partition function $Z_n(\IS,\alpha,\beta)$ is given by
\begin{align*}
\sum_{\sigma(1),\ldots ,\sigma(n)} V_{\sigma(1)}T_{\sigma(1),\sigma(2)}T_{\sigma(2),\sigma(3)}\cdots T_{\sigma(n-1),\sigma(n)} &= V(T)^{n-1}W\\
&=a_+(\lambda_+)^{n-1}+a_-(\lambda_-)^{n-1},
\end{align*}
where 
\begin{align*}
a_+ &= \cosh \alpha + \frac{\E^{\beta}\sinh^2\alpha + \E^{-\beta}}{\sqrt{\E^{2\beta}\sinh^2\alpha + \E^{-2\beta}}}\qquad;\qquad a_-= \cosh \alpha - \frac{\E^{\beta}\sinh^2\alpha + \E^{-\beta}}{\sqrt{\E^{2\beta}\sinh^2\alpha + \E^{-2\beta}}}\\
\lambda_+ &=\E^{\beta}\,\cosh \alpha + \sqrt{\E^{2\beta} \sinh^2\alpha + \E^{-2\beta} }\qquad;\qquad \lambda_- =\E^{\beta}\,\cosh \alpha - \sqrt{\E^{2\beta} \sinh^2\alpha + \E^{-2\beta} }.
\end{align*}
Indeed, $\lambda_{\pm}$ are the two eigenvalues of $T$, and $a_+$ and $a_-$ are obtained by identification of coefficients in the two formulas \begin{align*}
Z_1(\IS,\alpha,\beta)&=\E^{\alpha}+\E^{-\alpha}\\
Z_2(\IS,\alpha,\beta)&=\E^{2\alpha+\beta}+\E^{-2\alpha+\beta} + 2\E^{-\beta}. 
 \end{align*}
Then, the Laplace transform of $M_n$ is given by 
$$\esper_{\alpha,\beta}[\E^{zM_n}]=\frac{Z_n(\IS,\alpha+z,\beta)}{Z_n(\IS,\alpha,\beta)}.$$
In particular, 
$$
\esper_{\alpha,\beta}[M_n]=\left.\frac{\partial \esper[\E^{zM_n}]}{\partial z}\right|_{z=0} = \frac{\partial }{\partial \alpha} \log Z_n(\IS,\alpha,\beta)= n \,\frac{\E^{\beta}\sinh\alpha }{\sqrt{\E^{2\beta}\sinh^2 \alpha+\E^{-2\beta}}}+O(1).
$$
whence a formula for the (asymptotic) mean magnetization by spin: $$\overline{m}=\frac{\E^{\beta}\sinh\alpha}{\sqrt{\E^{2\beta} \sinh^2 \alpha + \E^{-2\beta}}}.$$\bigskip

A more precise Taylor expansion of $Z_n(\IS,\alpha+z,\beta)$ leads to the following:

\begin{theorem}\label{thm:modgaussising}
Under the Ising measure $\IS_{\alpha,\beta}$, $\frac{M_n-n\overline{m}}{n^{1/3}}$ converges in the complex mod-Gaussian sense with parameters $$t_n = n^{1/3}\,\frac{\E^{-\beta}\cosh \alpha}{(\E^{2\beta}\sinh^2\alpha+\E^{-2\beta})^{3/2}}$$
and limiting function
$$\psi(z)=\exp\left(-\frac{2\E^{\beta}\sinh^3 \alpha+(3\E^{\beta}-\E^{-3\beta})\sinh\alpha}{6(\E^{2\beta}\sinh^2\alpha+\E^{-2\beta})^{5/2}}\,z^3\right).$$
\end{theorem}

\begin{proof} In the following, we are dealing with square roots and logarithms of complex numbers, but each time in a neighborhood of $\R_+^*$, so there is no ambiguity in the choice of the branches of these functions. That said, it is easier to work with log-Laplace transforms:
\begin{align*}
\log \esper\!\left[\E^{z\,\frac{M_n-n\overline{m}}{n^{1/3}}}\right]&= \log Z_n\left(\IS,\alpha + \frac{z}{n^{1/3}},\beta\right)-\log Z_n(\IS,\alpha,\beta)-zn^{2/3}\overline{m}\\
\log Z_n(\IS,\alpha,\beta) &= \log a_+(\alpha,\beta) + (n-1) \log \lambda_+(\alpha,\beta) + o(1)\\
\log Z_n\left(\IS,\alpha+ \frac{z}{n^{1/3}},\beta\right) &= \log a_+\left(\alpha+ \frac{z}{n^{1/3}},\beta\right) + (n-1) \log \lambda_+\left(\alpha+ \frac{z}{n^{1/3}},\beta\right) + o(1)\\
&= \log a_+(\alpha,\beta) + (n-1) \log \lambda_+(\alpha,\beta)+zn^{2/3} \frac{\partial}{\partial \alpha}\left(\log \lambda_+(\alpha,\beta)\right)\\
&\quad + \frac{z^2n^{1/3}}{2} \frac{\partial^2}{\partial \alpha^2}\left(\log \lambda_+(\alpha,\beta)\right) + \frac{z^3}{6} \frac{\partial^3}{\partial \alpha^3}\left(\log \lambda_+(\alpha,\beta)\right) +o(1).
\end{align*}
Thus, it suffices to compute the first derivatives of $\log \lambda_+(\alpha,\beta)$ with respect to $\alpha$:
\begin{align*}
\log \lambda_+(\alpha,\beta)&=\log\left(\E^{\beta}\cosh\alpha + \sqrt{\E^{2\beta}\sinh^2 \alpha + \E^{-2\beta}}\right)\\
\frac{\partial}{\partial \alpha}\left(\log \lambda_+(\alpha,\beta)\right)&=\frac{\E^{\beta}\sinh\alpha}{\sqrt{\E^{2\beta} \sinh^2 \alpha + \E^{-2\beta}}}=\overline{m}\\
\frac{\partial^2}{\partial \alpha^2}\left(\log \lambda_+(\alpha,\beta)\right)&=\frac{\E^{-\beta}\cosh \alpha}{(\E^{2\beta}\sinh^2\alpha+\E^{-2\beta})^{3/2}}=\sigma^2\\
\frac{\partial^3}{\partial \alpha^3}\left(\log \lambda_+(\alpha,\beta)\right)&=-\frac{2\E^{\beta}\sinh^3 \alpha+(3\E^{\beta}-\E^{-3\beta})\sinh\alpha}{(\E^{2\beta}\sinh^2\alpha+\E^{-2\beta})^{5/2}}=K_3.
\end{align*}
We therefore get
$$\log \esper\!\left[\E^{z\,\frac{M_n-n\overline{m}}{n^{1/3}}}\right] = n^{1/3}\,\frac{\sigma^2\, z^2}{2}+\frac{K_3\,z^3}{6}+o(1).$$
\end{proof}
\bigskip

By using Formula \ref{eq:upbound}, this result leads to new estimates of moderate deviations for the probability $\proba_{\alpha,\beta}[M_n \geq n\overline{m}+n^{1/3}x]$. In the special case when $\alpha=0$, the limiting function $\psi(z)$ of Theorem \ref{thm:modgaussising} is equal to $1$, and one has to push the expansion of $\log Z_n(\IS,0,\beta)$ to order $4$ to get a meaningful mod-Gaussian convergence (the same phenomenon will occur in the case of the Curie-Weiss model):

\begin{theorem}\label{thm:modgaussising2}
Under the Ising measure $\IS_{0,\beta}$, $\frac{M_n}{n^{1/4}}$ converges in the complex mod-Gaussian sense with parameters $t_n=n^{1/2}\,\E^{2\beta}$ and limiting function 
$$\psi(z)=\exp\left(-\frac{3\E^{6\beta}-\E^{2\beta}}{24}\,z^4\right).$$
\end{theorem}
\bigskip
\bigskip

\section{Mod-Gaussian convergence in $\mathrm{L}^1$ and the Curie-Weiss model}\label{sec:l1}

In this Section, $(X_n)_{n \in \N}$ is a sequence of random variables with entire moment generating series $\esper[\E^{zX_n}]$, and we assume the following:\vspace{2mm}
\begin{enumerate}[(A)]
\item One has mod-Gaussian convergence of the Laplace transforms, \emph{i.e.}, there is a sequence $t_n \to +\infty$ and a function $\psi$ continuous on $\R$ such that
$$\psi_n(t)=\esper[\E^{tX_n}]\,\E^{-\frac{t_n\,t^2}{2}}$$
converges locally uniformly on $\R$ to $\psi(t)$.\label{hyp:laplacemod}\vspace{1mm}
\item Each function $\psi_n$, and their limit $\psi$ are in $\mathrm{L}^1(\R)$.\label{hyp:l1mod}\vspace{1mm}
\end{enumerate}
We denote $\proba_n$ the law of $X_n$, 
\begin{equation}
\qproba_n[dx]=\frac{\E^{\frac{x^2}{2t_n}}}{\esper\!\left[\E^{\frac{(X_n)^2}{2t_n}}\right]}\,\,\proba_n[dx],\label{eq:exponchange}
\end{equation}
and $Y_n$ a random variable under the new law $\qproba_n$. Note that hypothesis \eqref{hyp:l1mod} implies that $Z_n=\esper[\E^{(X_n)^2/2t_n}]$ is finite for all $n \in \N$. Indeed,
$$
\int_{\R} \psi_n(t)\,dt = \esper\!\left[\int_{\R} \E^{t X_n-\frac{t_n \,t^2}{2}}\,dt\right]=\esper\!\left[\E^{\frac{(X_n)^2}{2t_n}}\,\left(\int_\R \E^{-\frac{(X_n-t_n t)^2}{2t_n}}\,dt\right)\right]=\sqrt{\frac{2\pi}{t_n}}\,\,\esper\!\left[\E^{\frac{(X_n)^2}{2t_n}}\right].
$$
Therefore the new probability measures $\qproba_n$ are well defined. The goal of this section is  to study the asymptotics of the new sequence $(Y_n)_{n \in \N}$. As we shall see in \S\ref{subsec:curieweiss}, the Curie-Weiss model defined by Equation \eqref{eq:curiemeasure} is one of the main examples in this framework. However, it is more convenient to look at the general problem, and we shall introduce later other models concerned by our general results.
\medskip

\subsection{Ellis-Newman lemma and deconvolution of a large Gaussian noise}
Suppose for a moment that  hypothesis \eqref{hyp:laplacemod} is replaced by the stronger hypotheses of Definition \ref{def:mod2}, with $c=+\infty$ and therefore $\mathcal{B}_c=\CC$. Fix then $0<a<b$, and consider the partial integral $\esper[\E^{(X_n)^2/2t_n}\,\mathbbm{1}_{t_na \leq X_n \leq t_nb}]$. By integration by parts of Riemann-Stieltjes integrals, one has:
\begin{align*}
\int_{t_na}^{t_nb}\E^{\frac{x^2}{2t_n}}\,\proba_n[dx] &= \left[-\E^{\frac{x^2}{2t_n}} \,\proba_n[X_n \geq x]\right]_{t_na}^{t_nb} + \int_{t_na}^{t_nb} \frac{x}{t_n}\,\E^{\frac{x^2}{2t_n}}\,\proba_n[X_n \geq x]\,dx\\
&=\left[-\E^{\frac{t_n\, x^2}{2}} \,\proba_n[X_n \geq t_nx]\right]_{a}^{b} + \int_{a}^{b} t_n x\,\E^{\frac{t_n\,x^2}{2}}\,\proba_n[X_n \geq t_n x]\,dx\\
&= \left(\left[-\frac{\psi(x)}{\sqrt{2\pi t_n}\,x}\right]_a^b +\sqrt{\frac{t_n}{2\pi}} \int_a^b \psi(x)\,dx\right)(1+o_{a,b}(1))\end{align*}
\begin{align*}
&=\left(\sqrt{\frac{t_n}{2\pi}} \int_a^b \psi(x)\,dx\right)(1+o_{a,b}(1))
\end{align*}
because of the estimates of precise deviations \eqref{eq:upbound}. In this computation, $o_{a,b}(1)$ is uniform for $a,b$ in compact sets of $(0,+\infty)$. In fact this estimate remains true for $a,b$ in a compact set of $\R$; hence, $a$ and $b$ can be possibly negative. If the  estimate is also true with $a=-\infty$ and $b=+\infty$, then 
\begin{align*}\qproba_n[t_na \leq Y_n \leq t_nb]&=\frac{\esper[\E^{(X_n)^2/2t_n}\,\mathbbm{1}_{t_na \leq X_n \leq t_nb}]}{\esper[\E^{(X_n)^2/2t_n}]}\\
& = \frac{\sqrt{\frac{t_n}{2\pi}}\,\int_a^b \psi(x)\,dx}{\sqrt{\frac{t_n}{2\pi}}\,\int_{-\infty}^{+\infty} \psi(x)\,dx}\, (1+o(1))\\
&= \frac{\int_a^b \psi(x)\,dx}{\int_{-\infty}^{+\infty} \psi(x)\,dx} \,(1+o(1)),
\end{align*}
so $(\frac{Y_n}{t_n})_{n \in \N}$ converges in law to the density $\psi(x)/\int_{\R}\psi(x)\,dx$.\bigskip

We now wish to identify the most general conditions under which this convergence in law happens. To this purpose, it is useful to produce random variables with density $\psi_n(x)/\int_{\R}\psi_n(x)\,dx$. They are given by the following Proposition, which appeared in \cite{EN78} as Lemma 3.3:

\begin{proposition}\label{prop:enlemma}
Let $G_n$ be a centred Gaussian variable with variance $\frac{1}{t_n}$, and independent from $Y_n$. The law of $W_n=G_n+\frac{Y_n}{t_n}$ has density $\psi_n(x)/\int_{\R}\psi_n(x)\,dx$.
\end{proposition}

\begin{proof}
Denote $Z_n=\esper[\E^{(X_n)^2/2t_n}]$, and $f_X(x)\,dx$ (respectively, $\proba_X$) the density (respectively, the law) of a random variable $X$. One has
\begin{align*}
\proba[W_n \leq w]&= \int_{- \infty}^w \left( \int_{\R} f_{G_n}(x-u)\,\proba_{\frac{Y_n}{t_n}}[du]\right)\,dx\\
&=\sqrt{\frac{t_n}{2\pi}}\int_{- \infty}^w \left( \int_{\R} \E^{-\frac{t_n\,(x-\frac{y}{t_n})^2}{2}}\,\proba_{Y_n}[dy]\right)\,dx\\
&=\sqrt{\frac{t_n}{2\pi}}\int_{- \infty}^w \left( \int_{\R} \E^{yx-\frac{y^2}{2t_n}}\,\qproba_{n}[dy]\right)\,\E^{-\frac{t_n\,x^2}{2}}\,dx\\
&=\frac{1}{Z_n}\,\sqrt{\frac{t_n}{2\pi}}\int_{- \infty}^w \left( \int_{\R} \E^{yx}\,\proba_{n}[dy]\right)\,\E^{-\frac{t_n\,x^2}{2}}\,dx\\
&=\frac{1}{Z_n}\,\sqrt{\frac{t_n}{2\pi}}\int_{- \infty}^w \psi_n(x)\,dx.
\end{align*}
Making $w$ go to $+\infty$ gives an equation for $Z_n=\sqrt{\frac{t_n}{2\pi}}\int_{\R} \psi_n(x)\,dx$. One concludes that:
$$\proba[W_n \leq w] = \frac{\int_{-\infty}^w \psi_n(x)\,dx}{\int_{-\infty}^\infty \psi_n(x)\,dx}.$$
\end{proof}
This important property was not used in our previous works: to get the residue of deconvolution $\psi_n$ of a random variable $X_n$ by a large Gaussian variable of variance $t_n$ (that is to say that one wants to \emph{remove} a Gaussian variable of variance $t_n$ from $X_n$), one can make the exponential change of measure \eqref{eq:exponchange}, and \emph{add} an independent Gaussian variable of variance $t_n$: the random variable thus obtained, which is $t_nW_n$ with the previous notation, has density proportional to $\psi_n(w/t_n)\,dw$.
\medskip

\subsection{The residue of mod-Gaussian convergence as a limiting law}
We can now state and prove the main result of this Section. We assume the hypotheses \eqref{hyp:laplacemod} and \eqref{hyp:l1mod}, and keep the same notation as before.

\begin{theorem}\label{thm:mainl1}
The following assertions are equivalent:\vspace{2mm}
\begin{enumerate}[(i)]
\item The sequence $(\frac{Y_n}{t_n})_{n \in \N}$ is tight.\label{l1i}\vspace{1mm}
\item The sequence $(\frac{Y_n}{t_n})_{n \in \N}$ converges in law to a variable with density $\psi(x)/\int_{\R}\psi(x)\,dx$.\label{l1ii}\vspace{1mm}
\item The convergence $\psi_n \to \psi$, which is supposed locally uniform on $\R$, also occurs in $\mathrm{L}^1(\R)$.\label{l1iii}\vspace{1mm}
\end{enumerate}
We shall then say that $(X_n)_{n \in \N}$ converges in the $\mathrm{L}^1$-mod-Gaussian sense with parameters $t_n$ and limiting function $\psi$. In this setting, the residue $\psi$ can be interpreted as the limiting law of $(X_n)_{n \in \N}$ after an appropriate change of measure.
\end{theorem}

\begin{proof}
Since the Gaussian variable $G_n$ of variance $\frac{1}{t_n}$ converges in probability to $0$, $(\frac{Y_n}{t_n})_{n \in \N}$ converges to a law $\mu$ if and only if $(W_n)_{n \in \N}$ converges to the law $\mu$. If \eqref{l1iii} is satisfied, then by Proposition \ref{prop:enlemma},
$$\lim_{n \to \infty} \proba[W_n \leq w] = \frac{\int_{-\infty}^w \psi(x)\,dx}{\int_{-\infty}^\infty \psi(x)\,dx},$$
so the cumulative distribution functions of the variables $W_n$ converge to the cumulative distribution function of the law $\psi(x)/\int_{\R}\psi(x)\,dx$, and \eqref{l1ii} is established. Obviously, one  also has \eqref{l1ii}$\,\Rightarrow\,$\eqref{l1i}. Finally, if \eqref{l1iii} is not satisfied, then by Scheffe's lemma one  also has
$$\int_{\R}\psi_n(x)\,dx\not\to  \int_{\R}\psi(x)\,dx.$$
However, by Fatou's lemma, $\int_{\R} \psi(x)\,dx \leq \liminf_{n \to \infty} \int_{\R}\psi_n(x)\,dx$. Therefore, the non-convergence in $\mathrm{L}^1$ is only possible if $\int_{\R} \psi(x)\,dx < \limsup_{n \to \infty} \int_{\R}\psi_n(x)\,dx$. Thus, there is an $\eps>0$ and a subsequence $(n_k)_{k \in \N}$ such that
$$\forall k \in \N,\,\,\int_{\R} \psi_{n_k}(x)\,dx \geq \eps + \int_\R \psi(x)\,dx.$$
Then, for all $a,b \in \R$,
\begin{align*}\limsup_{k \to \infty} \,\proba[a \leq W_{n_k} \leq b] &= \limsup_{k \to \infty}\left(\frac{\int_a^b \psi_{n_k}(x)\,dx}{\int_{\R} \psi_{n_k}(x)\,dx} \right) = \frac{\int_a^b \psi(x)\,dx}{\liminf_{k \to \infty} \int_\R \psi_{n_k}(x)\,dx}\\
& \leq \frac{\int_{\R} \psi(x)\,dx}{\eps+\int_{\R} \psi(x)\,dx}<1
\end{align*}
which amounts to saying that $(W_n)_{n \in \N}$ (and therefore $(\frac{Y_n}{t_n})_{n \in \N}$) is not tight; hence, \eqref{l1i} implies \eqref{l1iii}.
\end{proof}
\bigskip

To complete this result, it is important to compare the two notions of \emph{complex} mod-Gaussian convergence and of \emph{integral} $\mathrm{L}^1$-mod-Gaussian convergence. Though there are no direct implication between these two assumptions, the following Proposition shows that the latter notion is a stronger type of convergence:
\begin{proposition}\label{l1stronger}
Let $(X_n)_{n \in \N}$ be a sequence that converges in the $\mathrm{L}^1$-mod-Gaussian sense with parameters $t_n \to \infty$ and limiting function $\psi \in \mathrm{L}^1(\R)$. The estimate of precise large deviations \eqref{eq:upbound} is then satisfied.
\end{proposition}
\begin{proof}
Recall than $Z_n=\esper[\E^{(X_n)^2/2t_n}]=\sqrt{\frac{t_n}{2\pi}}\,\int_{\R}\psi_n(x)\,dx$. We want to compute
$$ \proba[X_n \geq t_n x] = \int_{t_n x}^\infty \proba_n[dy] = Z_n\,\int_{t_n x}^\infty \E^{-\frac{y^2}{2t_n}}\,\qproba_n[dy] =Z_n\,\int_{x}^\infty \E^{-\frac{t_n \,u^2}{2}}\,\proba_{\frac{Y_n}{t_n}}[du].
$$
Suppose for a moment that we can replace the law of $\frac{Y_n}{t_n}$ by the one of $W_n=G_n+\frac{Y_n}{t_n}$ in the previous computation. Then, one obtains from Proposition \ref{prop:enlemma}
$$
Z_n\,\int_{x}^\infty \E^{-\frac{t_n \,u^2}{2}}\,\proba_{W_n}[du]= \sqrt{\frac{t_n}{2\pi}}\,\int_{x}^\infty \E^{-\frac{t_n u^2}{2}}\psi_n(u)\,du.
$$
Fix $\eps>0$. Since $\psi_n$ converges locally uniformly to the continuous function $\psi$, there is an interval $[x,x+\eta]$ such that for $n$ large enough and $u \in [x,x+\eta]$, 
$$\psi(x)-\eps < \psi_n(u)<\psi(x) + \eps.$$
 Therefore, for $n$ large enough,
\begin{align*}
(\psi(x)- \eps)\,\int_{x}^{x+\eta} \E^{-\frac{t_n u^2}{2}}\,du &\leq \int_{x}^{x+\eta} \E^{-\frac{t_n u^2}{2}}\psi_n(u)\,du \leq (\psi(x)+ \eps)\,\int_{x}^{x+\eta} \E^{-\frac{t_n u^2}{2}}\,du\\
\downarrow \qquad& \qquad\qquad\qquad\qquad\qquad\qquad\qquad\qquad\qquad\qquad \downarrow \\
(\psi(x)- \eps) \,\frac{\E^{-\frac{t_n x^2}{2}}}{t_n x} &\qquad\qquad\qquad\qquad\qquad\qquad\qquad\qquad(\psi(x)+ \eps) \,\frac{\E^{-\frac{t_n x^2}{2}}}{t_n x}.
\end{align*}
Indeed, by integration by parts, $\int_x^{x+\eta} \E^{-\frac{t_n\,u^2}{2}}\,du$ is asymptotic to $\frac{\E^{-\frac{t_n\,x^2}{2}}}{t_n\,x}$. On the other hand, since $\psi_n \to_{\mathrm{L}^1} \psi$, for the remaining part of the integral,
$$\int_{x+\eta}^\infty \E^{-\frac{t_n u^2}{2}}\,\psi_n(u)\,du  \leq \E^{-\frac{t_n (x+\eta)^2}{2}}\left(\int_{x+\eta}^\infty \,\psi_n(u)\,du\right) \simeq \E^{-\frac{t_n (x+\eta)^2}{2}}\left(\int_{x+\eta}^\infty \,\psi(u)\,du\right) $$
which is much smaller than the previous quantities. Therefore, assuming that one can replace $\frac{Y_n}{t_n}$ by $W_n$, we obtain the asymptotics
$$\proba[X_n \geq t_n x] = \frac{\E^{-\frac{t_n\,x^2}{2}}}{\sqrt{2\pi t_n} \,x}\,\psi(x)\,(1+o(1))$$
for all $x>0$; this is what we wanted to prove. Finally, the replacement $\frac{Y_n}{t_n}\leftrightarrow W_n$ is indeed valid, because
\begin{align*}
\int_x^\infty \E^{-\frac{t_n\,u^2}{2}}\,\proba_{W_n}[du] &= \left[\E^{-\frac{t_n\,u^2}{2}}\,\proba[W_n \leq u]\right]_x^\infty +\int_x^\infty t_n u\,\E^{-\frac{t_n\,u^2}{2}}\,\proba[W_n \leq u]\,du \\
&\simeq \left[\E^{-\frac{t_n\,u^2}{2}}\,\proba[Y_n/t_n \leq u]\right]_x^\infty + \int_x^\infty t_n u\,\E^{-\frac{t_n\,u^2}{2}}\,\proba[Y_n/t_n \leq u]\,du\\
&\simeq \int_x^\infty \E^{-\frac{t_n\,u^2}{2}}\,\proba_{\frac{Y_n}{t_n}}[du] 
\end{align*}
by using on the second line the fact that both $\frac{Y_n}{t_n}$ and $W_n$ converge in law to the same limit, and therefore have equivalent cumulative distribution function on $\R_+$.
\end{proof}
\bigskip

In the same setting of $\mathrm{L}^1$-mod-Gaussian convergence, one has similarly the estimates on the negative part of the real line, and around $0$, as described on page \pageref{eq:upbound} in the setting of complex mod-Gaussian convergence.\medskip

\subsection{Application to the Curie-Weiss model}\label{subsec:curieweiss}
Consider i.i.d.~Bernoulli random variables $(\sigma(i))_{i \geq 1}$ with $\proba[\sigma(i)=1]=1-\proba[\sigma(i)=-1]=\frac{\E^{\alpha}}{2\,\cosh\alpha}$ for some $\alpha \in\R$. We set $U_n=\sum_{i=1}^n \sigma(i)$, so that
\begin{align*}
\esper[\E^{zU_n}]&=\left(\frac{\cosh (z+\alpha)}{\cosh \alpha}\right)^n=\left(\cosh z + \sinh z \,\tanh \alpha\right)^n \\
\esper\!\left[\E^{z\,\frac{U_n-n\,\tanh\alpha}{n^{1/3}}}\right]&=\left(\frac{\cosh (zn^{-1/3}) + \sinh (zn^{-1/3}) \,\tanh \alpha}{\E^{zn^{-1/3}\,\tanh \alpha}}\right)^n\\
\log \esper\!\left[\E^{z\,\frac{U_n-n\,\tanh\alpha}{n^{1/3}}}\right] &=\frac{n^{1/3}}{2\,\cosh^2 \alpha}\,z^2-\frac{\sinh \alpha}{3 \cosh^3 \alpha}\,z^3+o(1)
\end{align*}
so one has complex mod-Gaussian convergence of $\frac{U_n-n\tanh \alpha}{n^{1/3}}$ with parameters $\frac{n^{1/3}}{\cosh^2\alpha}$ and limiting function $\exp(-\frac{\sinh \alpha}{3\cosh^3 \alpha}\,z^3)$. \bigskip

If $\alpha=0$, then the term of order $3$ disappears in the Taylor expansion of the characteristic function, and one obtains instead
$$\log \esper\!\left[\E^{z\,\frac{U_n}{n^{1/4}}}\right] =\frac{n^{1/2}\,z^2}{2}-\frac{z^4}{12}+o(1),$$
hence a complex mod-Gaussian convergence of $X_n=\frac{U_n}{n^{1/4}}$ with parameters $n^{1/2}$ and limiting function $\exp(-z^4/12)$. Since this function restricted to $\R$ is integrable, this leads us to the following result, which originally appeared in \cite{EN78} (without the mod-Gaussian interpretation):

\begin{theorem}\label{thm:ellnew}
Let $X_n=n^{-1/4}\sum_{i=1}^n \sigma(i)$ be a rescaled sum of centred $\pm 1$ independent Bernoulli random variables. It converges in the $\mathrm{L}^1$-mod-Gaussian sense, with parameters $n^{1/2}$ and limiting function $\exp(-\frac{z^4}{12})$. As a consequence, if $Y_n=n^{-1/4}M_n$ is the rescaled magnetization of a Curie-Weiss model $\CW_{0,1}$ of parameters $\alpha=0$ and $\beta=1$, then $Y_n/n^{1/2}$ converges in law to the distribution
$$\frac{\exp(-\frac{x^4}{12})\,dx}{\int_{\R}\exp(-\frac{x^4}{12})\,dx}.$$
\end{theorem}

\begin{proof}
The function $\psi_n(t)$ is in our case 
$$\psi_n(t)=\E^{-\frac{t^2n^{1/2}}{2}}\,\left(\cosh \frac{t}{n^{1/4}}\right)^n,$$
and we have seen that it converges locally uniformly to $\psi(t)=\exp(-\frac{t^4}{12})$. By Scheffe's lemma, to obtain the $\mathrm{L}^1$-mod-convergence, it is sufficient to prove that $\int_{\R} \psi_n(t)\,dt$ converges to $\int_{\R}\exp(-\frac{t^4}{12})\,dt$. This is a simple application of Laplace's method:
\begin{align*}
\int_{\R} \psi_n(t)\,dt &= \int_{\R}\E^{-\frac{t^2n^{1/2}}{2}}\,\left(\cosh \frac{t}{n^{1/4}}\right)^n\,dt = n^{1/4}\int_{\R} \left(\E^{-\frac{u^2}{2}}\,\cosh u\right)^n\,du
\end{align*}
and the function $u \mapsto \E^{-\frac{u^2}{2}}\,\cosh u$ attains its global maximum at $u=0$, with a Taylor expansion $1-\frac{u^4}{12}+o(u^4)$, see Figure \ref{fig:functionellnew}.
\begin{center}
\begin{figure}[ht]
\includegraphics[scale=0.5]{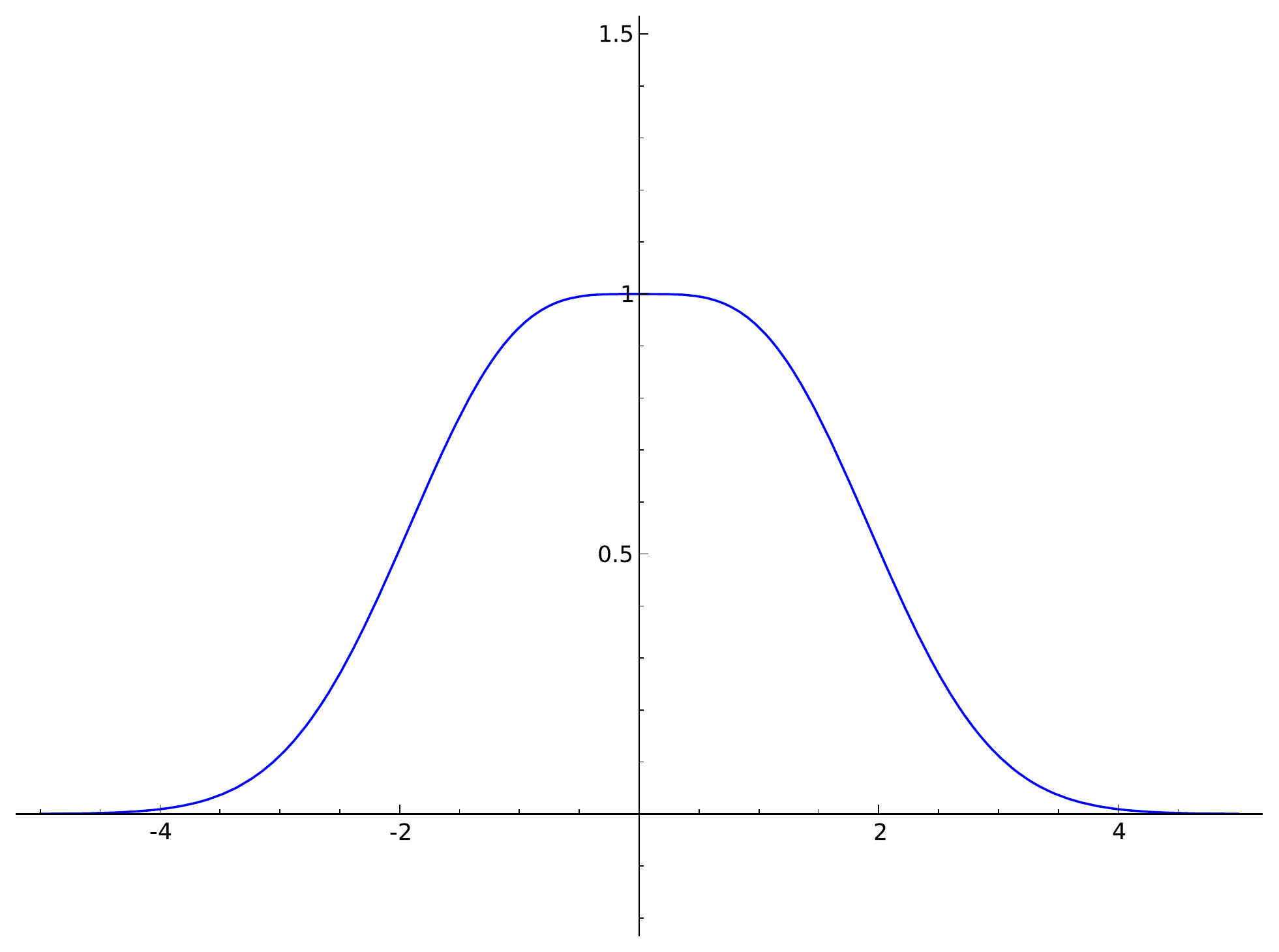}
\caption{The function $f(u)=\E^{-\frac{u^2}{2}}\,\cosh u$. \label{fig:functionellnew}}
\end{figure}
\end{center}
Then, the exponential change of measure \eqref{eq:exponchange} gives a probability measure on spin configurations proportional to
$$\exp\left(\frac{(Y_n)^2}{2n^{1/2}}\right)=\exp\left(\frac{1}{2n}\,(M_n)^2\right),$$
so it is indeed the Curie-Weiss model $\CW_{0,1}$.
\end{proof}
\bigskip

It is easily seen that the proof adapts readily to the case where Bernoulli variables are replaced by so-called pure measures, so we recover all the limit theorems stated in \cite{EN78,ENR80}. However, by choosing the setting of mod-Gaussian convergence, we also obtain new limit theorems for models that do not fall in the Curie-Weiss setting. The following result explains how it would work to replace the Bernoulli distribution by more general ones; \emph{cf.} \cite[Proposition 2.2]{KNN13}. 

\begin{proposition}
Let $k\geq2$ be an integer, and let $(B_n)_{n\geq1}$ be a sequence of i.i.d random variables in $\mathrm{L}^r$ for some $r>k+1$, such that the first $k$ moments of $B_1$ are the same as the corresponding moments of the Standard Gaussian distribution. Then the sequence of random variables 
$$\left(\frac{1}{n^{1/(k+1)}} \sum_{k=1}^n B_k\right)_{n\geq1}$$
converges in the mod-Gaussian sense with parameters $$t_n=n^{(k-1)/(k+1)},$$
and limiting function $$\theta(t)=\E^{(\I t)^{k+1}\frac{c_{k+1}}{(k+1)!}},$$
where $c_{k+1}$ denotes the $(k+1)$-th cumulant of $B_1$.
\end{proposition}

When the random variables $B_n$ have an entire moment generating function, then one can replace $t$ with $-\I t$ to obtain mod-Gaussian convergence with the Laplace transforms. If $B_1$ is  symmetric, then $k$ is necessarily an 
odd number of the form $2s-1$ and hence $$\psi(t)=\E^{(-1)^s t^{2s}\frac{c_{2s}}{(2s)!}}.$$In the case of the Bernoulli random variables, $s=2$ and $c_4=-1/12$. In order to have our theorem of $\mathrm{L}^1$-mod-Gaussian convergence to hold, we need to find conditions on the distribution of $B_1$ such that $c_{2s}$ is negative and that $\int_\R \psi_n$ converges to $\int_\R \psi$. The conditions in \cite{EN78} and \cite{ENR80} precisely imply these. But within our more general framework, following the discussion in Section \ref{subsec:methjoint}, we could well imagine a situation which fulfils the assumptions of Theorem \ref{thm:mainl1} but where the initial symmetric random variables are not necessarily i.i.d but simply independent or even weakly dependent. The following paragraph yields an example of such a setting.
\medskip

\subsection{Mixed Curie-Weiss-Ising model}

Consider the one-dimensional Ising model of parameter $\alpha=0$, and $\beta$ arbitrary. We have shown in Section \ref{sec:ising} the complex mod-Gaussian convergence of $(n^{-1/4}\,M_n)_{n \in \N}$ with parameters $n^{1/2}\,\E^{2\beta}$ and limiting function $\psi(z)=\exp(-(3\E^{6\beta}-\E^{2\beta})\,z^4/24)$. Restricted to $\R$, this limiting function is integrable, and again one has $\mathrm{L}^1$-mod-convergence. Indeed, recall that 
$$\esper[\E^{t M_n}]=\frac{Z_n(\IS,t,\beta)}{Z_n(\IS,0,\beta)}=\frac{1}{2}\left(a_+(t,\beta)\left(\frac{\lambda_+(t,\beta)}{2\cosh \beta}\right)^{n-1}+a_-(t,\beta)\left(\frac{\lambda_-(t,\beta)}{2\cosh \beta}\right)^{n-1}\right).$$
It will be convenient to work with $n^{-1/4}\,M_{n+1}$ instead of $n^{-1/4}\,M_n$ in order to work with $n$-th powers. Then,
\begin{align*}
\psi_n(t)&=\esper\!\left[\E^{t \frac{M_{n+1}}{n^{1/4}}}\right]\,\E^{-\frac{n^{1/2}\E^{2\beta}z^2}{2}} \\
\int_{\R}\psi_n(t) dt& = \frac{n^{1/4}}{2}\int_{\R} a_+(u,\beta) \left(\frac{\lambda_+(u,\beta)}{2\cosh \beta}\,\E^{-\frac{\E^{2\beta}u^2}{2}}\right)^n + a_-(u,\beta) \left(\frac{\lambda_-(u,\beta)}{2\cosh \beta}\,\E^{-\frac{\E^{2\beta}u^2}{2}}\right)^n du
\end{align*}
and for every parameter $\beta \geq 0$, the functions 
$$u \mapsto \frac{\lambda_+(u,\beta)}{2\cosh \beta}\,\E^{-\frac{\E^{2\beta}u^2}{2}} \quad \text{and} \quad u \mapsto \frac{\lambda_-(u,\beta)}{2\cosh \beta}\,\E^{-\frac{\E^{2\beta}u^2}{2}} $$ 
attain their unique maximum at $u=0$, see Figure \ref{fig:function3} for the graph of the first function. 
\begin{center}
\vspace{-5mm}
\begin{figure}[ht]
\includegraphics[scale=0.7]{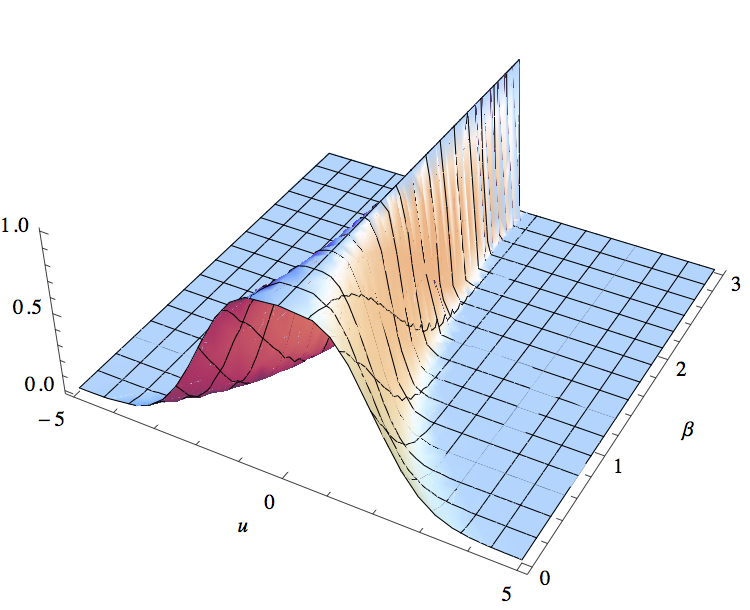}\caption{The function $f(u,\beta)=\frac{\lambda_+(u,\beta)}{2\cosh \beta}\,\E^{-\frac{\E^{2\beta}u^2}{2}}$ (using \textsc{Mathematica}). \label{fig:function3}}
\end{figure}
\end{center}
Their Taylor expansions at $u=0$ are respectively
$$1-\frac{3\E^{6\beta}-\E^{2\beta}}{24}\,u^4+o(u^4)\quad\text{and} \quad \tanh \beta +o(1),$$
so again by the Laplace method we get $\lim_{n \to \infty} \int_{\R}\psi_n(t)\,dt = \int_{\R} \psi(t)\,dt$ and the $\mathrm{L}^1$-mod-convergence. As a consequence, consider the random configuration of spins $\sigma$ on $\lle 1,n\rre$ with probability proportional to
$$\exp\left(\beta \left(\sum_{i=1}^{n-1} \sigma(i)\sigma(i+1)\right)+\frac{1}{2n\E^{2\beta}} \left(\sum_{i=1}^n \sigma(i)\right)^2\right).$$
This model has a local interaction with coefficient $\beta$ and a global interaction with coefficient $\frac{1}{\E^{2\beta}}$, so it is a mix of the Ising model and of the Curie-Weiss model. The previous discussion and Theorem \ref{thm:mainl1} show that its magnetization satisfies the non standard limit theorem
$$\frac{M_n}{n^{3/4}} \rightharpoonup_{n \to \infty} \frac{\psi(x)\,dx}{\int_{\R}\psi(x)\,dx}\quad\text{with }\psi(x)=\exp\left(-\frac{3\E^{6\beta}-\E^{2\beta}}{24}\,x^4\right).$$
\medskip

\subsection{Sub-critical changes of measures}
In the mixed Curie-Weiss-Ising model, one may ask  what happens if instead of $\beta$ and $\frac{1}{\E^{2\beta}}$ one puts arbitrary coefficients for the local and the global interaction. More generally, given a sequence $(X_n)_{n \in \N}$ that converges in the $\mathrm{L}^{1}$-mod-Gaussian sense with parameters $t_n$ and limiting function $\psi$, one can look at the change of measure
$$\qproba_{n}^{(\gamma)}[dx]=\frac{\E^{\frac{\gamma x^2}{2t_n}}}{\esper\!\left[\E^{\frac{\gamma (X_n)^2}{2t_n}}\right]}\,\proba_n[dx]$$
with $\gamma \in (0,1)$ (for $\gamma >1$, the change of measure is not necessarily well-defined, since the hypotheses \eqref{hyp:laplacemod} and \eqref{hyp:l1mod} do not ensure that $\esper[\E^{\gamma (X_n)^2/2t_n}] < + \infty$). These subcritical changes of measures do not modify the order of magnitude of the fluctuations of $X_n$, and more precisely:

\begin{theorem}\label{thm:subcritical}
Suppose that $(X_n)_{n\in \N}$ converges in the $\mathrm{L}^1$-mod-Gaussian sense with parameters $t_n$ and limiting function $\psi$. Then, if $(X_{n}^{(\gamma)})_{n \in \N}$ is a sequence of random variables under the new probability measures $\qproba_{n}^{(\gamma)}$, it converges in the $\mathrm{L}^1$-mod-Gaussian sense with parameters $(1-\gamma)t_n$ and limit $\psi$.
\end{theorem}

\begin{example}
Consider a random configuration of spins $\sigma$ on $\lle 1,n\rre$ with probability proportional to
$$\exp\left(\beta \left(\sum_{i=1}^{n-1} \sigma(i)\sigma(i+1)\right)+\frac{\gamma}{2n} \left(\sum_{i=1}^n \sigma(i)\right)^2\right),$$
with $\gamma < \E^{-2\beta}$. The total magnetization of the system has order of magnitude $n^{1/2}$, and more precisely, one has the central limit theorem
$$\frac{M_n}{n^{1/2}} \rightharpoonup_{n \to \infty} \mathcal{N}(0,(1-\gamma\E^{2\beta})\E^{2\beta}),$$
and in fact a $\mathrm{L}^1$-mod-Gaussian convergence of $\frac{M_n}{n^{1/4}}$, with a limiting function $$\psi(x)=\exp(-(3\E^{6\beta}-\E^{2\beta})\,x^4/24)$$ that does not depend on $\gamma$.
\end{example}

\begin{proof}[Proof of Theorem \ref{thm:subcritical}]
We denote as before $(Y_n)_{n \in \N}$ a sequence of random variables under the laws $\qproba_n=\qproba_n^{(1)}$. We first compute the asymptotics of $Z_n^{(\gamma)}=\esper[\E^{\gamma (X_n)^2/2t_n}] $:
\begin{align*}
Z_n^{(\gamma)}&= Z_n\,\,\esper[\E^{-(1-\gamma) (Y_n)^2/2t_n}]\end{align*}
\begin{align*}
&=\sqrt{\frac{t_n}{2\pi}}\left(\int_{\R}\psi(x)\,dx\right)\,\esper\!\left[\E^{-\frac{t_n(1-\gamma)}{2}(\frac{Y_n}{t_n})^2}\right](1+o(1))\\
&=\sqrt{\frac{t_n}{2\pi}}\left(\int_{\R}\psi(x)\,dx\right)\,\esper\!\left[\E^{-\frac{t_n(1-\gamma)}{2}\,(W_n)^2}\right](1+o(1))\\
&=\sqrt{\frac{1}{1-\gamma}}\,(1+o(1))
\end{align*}
by using on the third line the same argument as in the proof of Proposition \ref{l1stronger} to replace $\frac{Y_n}{t_n}$ by $W_n$; and by using the Laplace method on the fourth line in order to compute $\int_{\R} \E^{-t_n(1-\gamma)x^2/2}\,\psi_n(x)\,dx$. The same computations give the asymptotics of
\begin{align*}
\esper[\E^{tX_n+\gamma (X_n)^2/2t_n}]&=Z_n\,\esper[\E^{tY_n-(1-\gamma) (Y_n)^2/2t_n}]\\
&=\sqrt{\frac{t_n}{2\pi}}\left(\int_{\R}\psi(x)\,dx\right)\,\esper\!\left[\E^{t_nt\,(\frac{Y_n}{t_n})-\frac{t_n(1-\gamma)}{2}(\frac{Y_n}{t_n})^2}\right](1+o(1))\\
&=\sqrt{\frac{t_n}{2\pi}}\left(\int_{\R}\psi(x)\,dx\right)\,\esper\!\left[\E^{t_nt \,W_n-\frac{t_n(1-\gamma)}{2}(W_n)^2}\right](1+o(1))\\
&=\E^{-\frac{t_nt^2}{2(1-\gamma)}} \sqrt{\frac{1}{1-\gamma}}\,\psi(t)\,(1+o(1))
\end{align*}
with again a Laplace method on the fourth line. Since
$$\esper[\E^{tX_n^{(\gamma)}}]=\frac{\esper[\E^{tX_n+\gamma (X_n)^2/2t_n}]}{Z_n^{(\gamma)}},$$
this shows the hypotheses \eqref{hyp:laplacemod} and \eqref{hyp:l1mod} for the sequence $(X_n^{(\gamma)})_{n \in \N}$. Then, since $(\frac{Y_n}{t_n})_{n \in \N}$ converges in law, by using the implication \eqref{l1ii} $\Rightarrow$ \eqref{l1iii} in Theorem \ref{thm:mainl1} for the sequence  $(X_n^{(\gamma)})_{n \in \N}$, we see that the mod-Gaussian convergence of Laplace transforms necessarily happens in $\mathrm{L}^1(\R)$.
\end{proof}
\medskip

\subsection{Random walks changed in measure}\label{subsec:randomwalkchange}
In this section, we shall make a brief excursion in the higher dimensions. Since we do not want to enter details on mod-Gaussian convergence for random vectors (for which we refer the reader to \cite{KN12} and \cite{FMN13}), we shall only consider the simple case $X=(X^{(1)},\ldots,X^{(d)})$ is a random vector with values in $\R^d$ such that $\mathbb E[\exp(z_1 X^{(1)}+\cdots+z_d X^{(d)})]$ is entire in $\CC^d$. We shall say that the sequence $(X_n)$ of random vectors converges in the complex  mod-Gaussian sense with parameter $t_n$ and limiting function $\psi(z_1,\cdots,z_d)$ if the following convergence holds locally uniformly on compact subsets of $\CC^d$:
$$\psi_n(t)= \mathbb E[\exp(z_1 X_n^{(1)}+\cdots+z_d X_n^{(d)})]\exp\left(-t_n \frac{(z_1)^2+\cdots+(z_d)^2}{2}\right)\to \psi(z_1,\ldots,z_d).$$
In this vector setting, the assumptions (A) and (B) of Section \ref{sec:l1} now simply amount to the fact that the convergence above holds locally uniformly for $t=(t^{(1)},\cdots,t^{(d)})\in\R^d$ and that $\psi_n$ and $\psi$ are both in $\mathrm{L}^1(\R^d)$.\bigskip

Following the case $d=1$ we denote $\proba_n$ the law of $X_n$ on $\R^d$, 
$$
\qproba_n[dx]=\frac{\E^{\frac{\|x\|^2}{2t_n}}}{\esper\!\left[\E^{\frac{\|X_n\|^2}{2t_n}}\right]}\,\,\proba_n[dx],
$$
and $Y_n$ a random variable under the new law $\qproba_n$. Note that here again hypothesis \eqref{hyp:l1mod} implies that $Z_n=\esper[\E^{\|X_n\|^2/2t_n}]$ is finite for all $n \in \N$. Indeed, with the notation $\langle u,v\rangle=u_1v_1+\cdots+u_dv_d$, we have
\begin{align*}
\int_{\R^d} \psi_n(t)\,dt &= \esper\!\left[\int_{\R^d} \E^{\langle t, X_n\rangle-\frac{t_n \,\|t\|^2}{2}}\,dt\right] \\
&=\esper\!\left[\E^{\frac{(\|X_n\|^2}{2t_n}}\,\left(\int_{\R^d} \E^{-\frac{\|X_n-t_n t\|^2}{2t_n}}\,dt\right)\right]=\left(\frac{2\pi}{t_n}\right)^{d/2}\,\,\esper\!\left[\E^{\frac{\|X_n\|^2}{2t_n}}\right].
\end{align*}

Therefore, the new probabilities $\qproba_n$ are well-defined and $$Z_n=\esper[\E^{\|X_n\|^2/2t_n}]=\left(\frac{t_n}{2\pi}\right)^{d/2}\int_{\R^d} \psi_n(t)\,dt.$$
Then it is clear that Proposition \ref{prop:enlemma} holds with $G_n$ being a Gaussian vector with covariance matrix $1/t_n\; I_d$ where $I_d$ is the identity matrix of size $d$. Similarly one can establish an analogue of Theorem \ref{thm:mainl1} in $\R^d$.
\bigskip

Let $W_n$ be a simple random walk on the lattice $\Z^{d\geq 2}$: at each step, each of the $2d$ neighbors of the state that is occupied has the same probability of transition $(2d)^{-1}$. The $d$-dimensional characteristic function of $W_n=(W_n^{(1)},\ldots,W_n^{(d)})$ is
$$\esper[\E^{z_1W_n^{(1)}+\cdots+z_dW_n^{(d)}}]=\left(\frac{\cosh z_1 + \cdots + \cosh z_d}{d}\right)^n.$$
Therefore, one has the asymptotics 
\begin{align*}&\log\,\esper\!\left[\E^{\frac{z_1W_n^{(1)}+\cdots+z_dW_n^{(d)}}{n^{1/4}}}\right]\\
&=n \log\left(1+\frac{(z_1)^2+\cdots+(z_d)^2}{2dn^{1/2}} +\frac{(z_1)^4+\cdots+(z_d)^4}{24dn}+o\left(\frac{1}{n}\right) \right)\\
&=n^{1/2}\,\frac{(z_1)^2+\cdots+(z_d)^2}{2d} -\frac{3((z_1)^2+\cdots+(z_d)^2)^2-d((z_1)^4+\cdots+(z_d)^4)}{24d^2}+o(1).
\end{align*}
One obtains a $d$-dimensional complex mod-Gaussian convergence of $X_n=n^{-1/4}\,W_n$ with parameters $\frac{n^{1/2}}{d}$ and limiting function
$$\psi(z_1,\ldots,z_d)=\exp\left(-\frac{3\,((z_1)^2+\cdots+(z_d)^2)^2-d\,((z_1)^4+\cdots+(z_d)^4)}{24d^2}\right).$$
In \cite{FMN13}, we used this mod-convergence to prove quantitative estimates regarding the breaking of the radial symmetry when one considers random walks conditioned to be of large size (of order $n^{3/4}$ instead of the expected order $n^{1/2}$). With the notion of $\mathrm{L}^1$-mod-Gaussian convergence, one can give another interpretation, but only for $d=2$ or $d=3$. Restricted to $\R^d$, the limiting function is indeed not integrable for $d \geq 4$: if $t_2,\ldots,t_d \in [-1,1]$, then 
\begin{align*}
3\,((t_1)^2+\cdots+(t_d)^2)^2-d\,&((t_1)^4+\cdots+(t_d)^4) \leq 3\,((t_1)^2+(d-1))^2-d\,(t_1)^4\\
&\leq (3-d)(t_1)^4 + 6(d-1)(t_1)^2 + 3(d-1)^2.
\end{align*}
So, restricted to the domain $\R\times [-1,1]^{d-1}$, $\psi(t_1,\ldots,t_d) \leq K\,\exp(a (t_1)^4 - b (t_1)^2)$ for some positive constants $a$, $b$ and $K$; therefore, this function is not integrable.\bigskip

On the other hand, if $d=2$ or $d=3$, then $\psi$ is integrable on $\R^d$, and one has $\mathrm{L}^1$-mod-Gaussian convergence. Indeed, when $d=2$, the limiting function is 
\begin{equation}\psi(t_1,t_2)=\exp\left(-\frac{(t_1)^4+(t_2)^4 + 6 (t_1t_2)^2}{96}\right),\label{eq:limitwalk2}\end{equation}
which is clearly integrable; and the residues
$$\psi_n(t_1,t_2)=\esper\left[\E^{\frac{t_1W_n^{(1)} + t_2W_n^{(2)}}{n^{1/4}}}\right]\,\E^{-\frac{n^{1/2}((t_1)^2+(t_2)^2)}{4}}$$
converge locally uniformly on $\R^2$ to $\psi(t_1,t_2)$, but also in $\mathrm{L}^1(\R^2)$. Indeed,
\begin{align*}
\int_{\R^2} \psi_n(t_1,t_2)\,dt_1\,dt_2 &= \int_{\R^2} \left(\frac{\cosh \frac{t_1}{n^{1/4}}+\cosh \frac{t_2}{n^{1/4}}}{2}\right)^n \E^{-\frac{n^{1/2}((t_1)^2+(t_2)^2)}{4}}\,dt_1\,dt_2\\
&= n^{1/2} \int_{\R^2} \left(\frac{\cosh u_1 + \cosh u_2}{2}\,\E^{-\frac{(u_1)^2+(u_2)^2}{4}}\right)^n\,du_1\,du_2,
\end{align*}
and the function $(u_1,u_2)\mapsto \frac{\cosh u_1 + \cosh u_2}{2}\,\E^{-\frac{(u_1)^2+(u_2)^2}{4}}$ reaches its unique global maximum at $u_1=u_2=0$, with Taylor expansion $$1-\frac{(u_1)^4+(u_2)^4 + 6 (u_1u_2)^2}{96}+o(\|u\|^4)$$ around this point (see Figure \ref{fig:function2}). 

\begin{center}
\begin{figure}[ht]
\vspace{-8mm}
\includegraphics[scale=0.7]{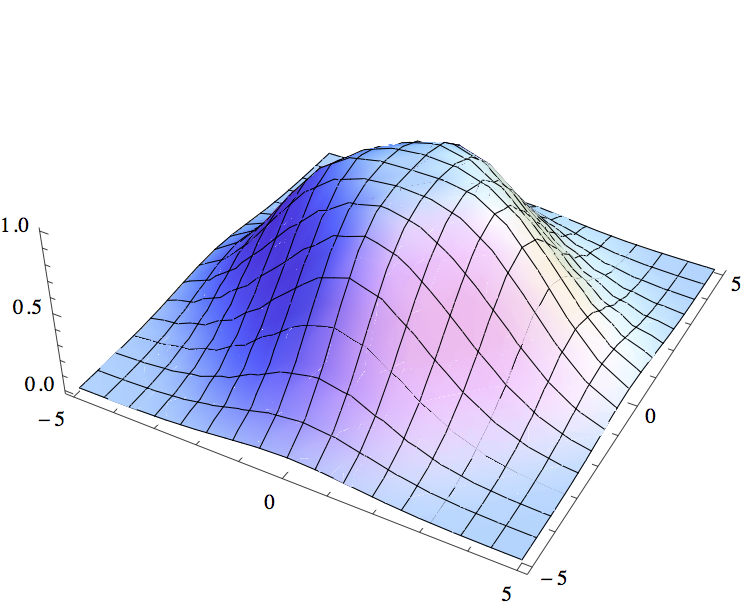}\caption{The function $f(u_1,u_2)=\frac{\cosh u_1 + \cosh u_2}{2}\,\E^{-\frac{(u_1)^2+(u_2)^2}{4}}$. \label{fig:function2}}
\end{figure}
\end{center}

Thus, by using the multi-dimensional Laplace method, one sees that the limit of the integral $\int_{\R^2} \psi_n(t_1,t_2)\,dt_1\,dt_2$ is $\int_{\R^2} \psi(t_1,t_2)\,dt_1\,dt_2$, and the $\mathrm{L}^1$ convergence is shown. Similarly, when $d=3$, the limiting function is
 \begin{equation}
\psi(t_1,t_2,t_3)=\exp\left(-\frac{(t_1t_2)^2+(t_1t_3)^2+(t_2t_3)^2}{36}\right),\label{eq:limitwalk3} 
 \end{equation}
 and the following computation shows that it is integrable:
 \begin{align*}
 \int_{\R^3}\psi(x,y,z)\,dx\,dy \,dz&= \int_{\R^2}\E^{-\frac{(yz)^2}{36}}\left(\int_{\R}\E^{-\frac{y^2+z^2}{36}\,x^2}dx\right)dy\,dz\\
 &=6\sqrt{\pi}\int_{\R^2}\frac{\E^{-\frac{(yz)^2}{36}}}{\sqrt{y^2+z^2}}\,dy\,dz \\
 &=3\sqrt{\pi}\int_{r=0}^\infty \int_{\theta=0}^{\pi} \E^{-\frac{r^4\,\sin^2 \theta}{144}}\,dr\,d\theta \end{align*}
\begin{align*}
 &=12\sqrt{3\pi}\int_{r=0}^\infty \E^{-r^4}\,dr \int_{\theta=0}^{\frac{\pi}{2}} \frac{d\theta}{\sqrt{\sin \theta}}< + \infty
 \end{align*}
since $\frac{1}{\sqrt{\sin \theta}}$ is integrable at $0$. On the other hand, the residues
$$\psi_n(t_1,t_2,t_3)=\esper\left[\E^{\frac{t_1W_n^{(1)} + t_2W_n^{(2)}+t_3W_n^{(3)}}{n^{1/4}}}\right]\,\E^{-\frac{n^{1/2}((t_1)^2+(t_2)^2+(t_3)^2)}{6}}$$
converge to $\psi(t_1,t_2,t_3)$ locally uniformly on $\R^3$ and in $\mathrm{L}^1(\R^3)$. Indeed, one has again
$$\int_{\R^3} \psi_n(t_1,t_2,t_3)\,dt = n^{1/2}\int_{\R^3} \left(\frac{\cosh u_1 + \cosh u_2 + \cosh u_3}{3}\,\E^{-\frac{(u_1)^2+(u_2)^2+(u_3)^2}{6}}\right)^n\,du$$
and the function in the brackets reaches its unique maximum at $u_1=u_2=u_3=0$, with Taylor expansion corresponding to the limiting function $\psi$ after application of the Laplace method.\bigskip

The multidimensional analogue of Theorem \ref{thm:mainl1} thus yields the following multidimensional extension of the limit theorem for the Curie-Weiss model:

\begin{theorem}
Let $W_n$ be a simple random walk in dimension $d \leq 3$. If $V_n$ is obtained from $W_n$ by a change of measure by the factor $\exp(d\,\|W_n\|^2/2n)$, then 
$$\frac{V_n}{n^{3/4}} \rightharpoonup_{n \to \infty} \frac{\psi(x)\,dx}{\int_{\R^3}\psi(x)\,dx},$$
where $\psi(x)=\exp(-x^4/12)$ in dimension $1$, and $\psi$ is given by Formulas \eqref{eq:limitwalk2} and \eqref{eq:limitwalk3} in dimension $2$ and $3$.
\end{theorem}
\medskip

\begin{remark}
Suppose $d=2$. Then, there is a limit in law not only for $\frac{V_n}{n^{3/4}}$, but in fact for the whole random walk $(\frac{V_k}{n^{3/4}})_{k \leq n}$, viewed as a random element of $\mathcal{C}(\R_+,\R^2)$ or of the Skorohod space $\mathcal{D}(\R_+,\R^2)$, see Figure \ref{fig:saw}.
\figcap{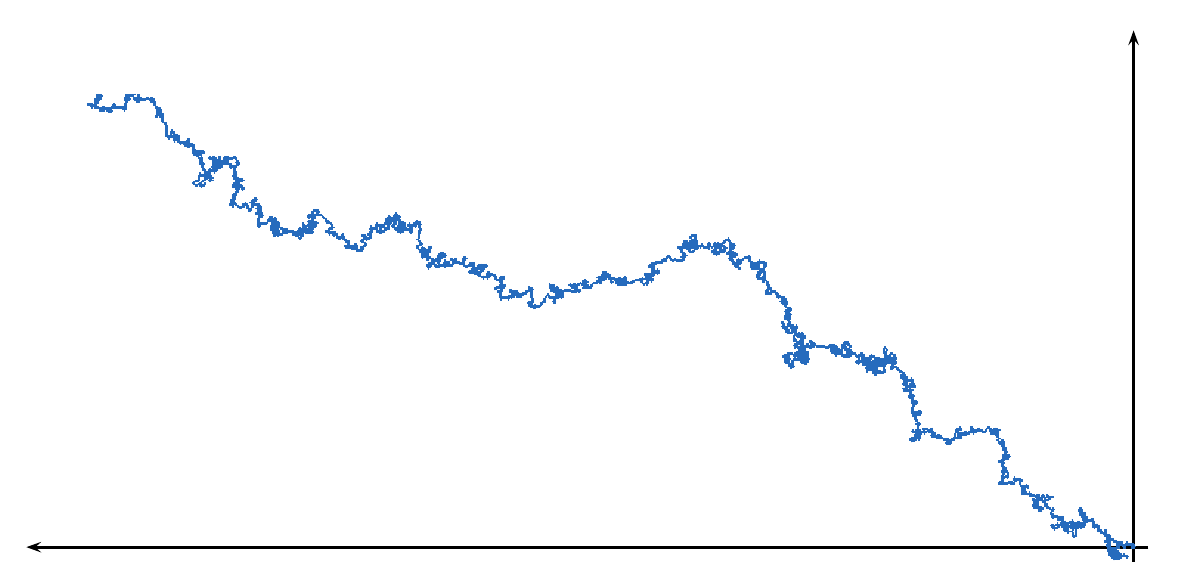}{A $2$-dimensional random walk changed in measure by $\E^{\|W_n\|^2/n}$, here with $n=10000$.\label{fig:saw}} 
\end{remark}
\bigskip
\bigskip

\section{Local limit theorem and rate of convergence in the Ellis-Newman limit theorem}\label{sec:rate}

We keep the same notation as before and note $I_n=\int_\R \psi_n(x)\,dx$ and $I_\infty=\int_\R \psi(x)\,dx$.
In this section we wish to provide a quick approach based on Fourier analysis,\vspace{2mm}

\begin{enumerate}
\item to compute the Kolmogorov distance between the rescaled magnetization $$Y_n/n^{1/2}=M_n/n^{3/4}$$ in the Curie-Weiss model and the random variable $W_\infty$ with density $\psi(x)/I_\infty$, where $\psi(x)=\exp(-x^4/12)$. This problem was recently solved in \cite{EL10} using Stein's method. As in \cite{EL10}, our method would cover many more general models as well: it is just a matter of specializing Lemma \ref{lem:kolmogorov} and Lemma \ref{lem:tao} below which are stated in all generality. \vspace{2mm}
\item to prove a new local limit theorem for the rescaled magnetization $n^{-1/4} M_n$ in the Curie-Weiss model. Here again we shall indicate how one can establish local limit theorems in more general situations. 
\end{enumerate}
\medskip

\subsection{Speed of convergence}\label{subsec:speedofconv}
Getting back to our special case of the Curie-Weiss model, we denote $X_n=\frac{1}{n^{1/4}}\sum_{i=1}^n B_i$ a scaled sum of $\pm 1$ independent Bernoulli random variables; $Y_n$ the random variable with modified law 
$$\qproba_n[dy]=\frac{\E^{\frac{y^2}{2n^{1/2}}}\,\proba_n[dy]}{\esper\!\left[\E^{\frac{(X_n)^2}{2n^{1/2}}}\right]};$$
$G_n$ an independent Gaussian random variable of variance $\frac{1}{n^{1/2}}$; and $W_n=\frac{Y_n}{n^{1/2}} + G_n$. It follows from the previous results that the law of $W_n$ has density
$$\frac{\psi_n(x)}{I_n}=\frac{1}{I_n}\,\E^{-\frac{n^{1/2}x^2}{2}}\,\left( \cosh \frac{x}{n^{1/4}}\right)^n,$$
which converges in $\leb^1$ towards the law $\frac{\psi(x)}{I_\infty}=\frac{1}{I_\infty}\,\E^{-\frac{x^4}{12}}$. We hence wish for an upper bound for the Kolmogorov distance between $\frac{Y_n}{n^{1/2}}$ and $W_\infty$. For this we shall need the following general lemmas.
\bigskip

\begin{lemma}\label{lem:kolmogorov}
Consider the two distributions $W_n = \frac{\psi_n(x)\,dx}{I_n}$ and $W_\infty = \frac{\psi(x)\,dx}{I_\infty}$. The Kolmogorov distance between them is smaller than
$$\frac{\|\psi-\psi_n\|_{\mathrm{L}^1}}{I_\infty}\,(1+o(1)).$$
\end{lemma}

\begin{proof}
Fix  $a \in \R$, and suppose for instance that $\int_\R \psi(x)\,dx \geq \int_{\R} \psi_n(x)\,dx$. We have
\begin{align*}
F_{W_n}(a)-F_{W_{\infty}}(a) &= \left(\frac{\int_{-\infty}^a \psi_n(x)\,dx }{I_n}-\frac{\int_{-\infty}^a \psi(x)\,dx }{I_n}\right) \!+\! \left(\frac{\int_{-\infty}^a \psi(x)\,dx }{I_n}-\frac{\int_{-\infty}^a \psi(x)\,dx }{I_\infty}\right) \\
&= \frac{\int_{-\infty}^a (\psi_n(x)-\psi(x))\,dx}{I_n} + \left(\int_{-\infty}^a \psi(x)\,dx\right) \frac{\int_{-\infty}^\infty (\psi(x)-\psi_n(x))\,dx}{I_\infty I_n}\\
&\leq -\frac{\int_{-\infty}^a (\psi(x)-\psi_n(x))\,dx}{I_n}+\frac{\int_{-\infty}^\infty (\psi(x)-\psi_n(x))\,dx}{I_n}\\
&\leq \frac{\int_{a}^\infty (\psi(x)-\psi_n(x))\,dx}{I_n} \leq \frac{\|\psi-\psi_n\|_{\mathrm{L}^1}}{I_n}.
\end{align*}
Writing $F_{W_{\infty}}(a)-F_{W_n}(a) = (1-F_{W_n}(a)) - (1-F_{W_\infty}(a))$, one sees that the inequality is in fact valid with an absolute value on the left-hand side. Since $I_n=I_{\infty} (1+o(1))$, this shows the claim. If $\int_{\R} \psi_n(x)\,dx \geq \int_{\R} \psi(x)\,dx$, it suffices to exchange the roles played by $\psi_n$ and $\psi$ to get the inequality.
\end{proof}\bigskip

The asymptotics of the $\leb^1$-norm $\|\psi-\psi_n\|_{\leb^1}$ in the Curie-Weiss model are computed as follows. Noting that one always has $\psi_n(x) \geq \psi(x)$, it suffices to compute
$$\int_{\R} \psi_n(x)\,dx = \int_{\R} \E^{-\frac{n^{1/2}\,x^2}{2}}\,\left(\cosh(x\,n^{-1/4})\right)^{n}\,dx=n^{1/4}  \int_{\R} \left(\E^{-\frac{u^2}{2}}\,\cosh(u)\right)^{n}\,du.$$
By the Laplace method (see \cite[Formula (19.17), p. 624-625]{Zor04}), the asymptotics of the integral is
$$n^{-\frac{1}{4}}\left(\frac{{12}^{1/4}\,\Gamma(\frac{1}{4})}{2}\right)+n^{-\frac{3}{4}}\left(\frac{12^{3/4}\,\Gamma(\frac{3}{4})}{10}\right)+\text{smaller terms}.$$
The first term corresponds to $I_\infty=\int_{\R} \psi(x)\,dx = \int_{\R}\E^{-x^4/12}\,dx$. As a consequence, 
$$\frac{\|\psi-\psi_n\|_{\mathrm{L}^1}}{I_\infty}=\frac{1}{n^{1/2}}\,\frac{\sqrt{12}\,\Gamma(\frac{3}{4})}{5\,\Gamma(\frac{1}{4})}\,(1+o(1)) .$$
\bigskip

The main work now consists in computing $d_{\mathrm{Kol}}(\frac{Y_n}{n^{1/2}},W_n)$. We start by a Lemma which is a variation of arguments used for i.i.d. random variables in \cite[p. 87]{Tao12}. In the following, given a function $f \in \mathrm{L}^1(\R)$, we write its Fourier transform $\widehat{f}(\xi)=\int_{\R}f(x)\,\E^{\I \xi x}\,dx$. Recall that the function
$$\upsilon(\xi)=\begin{cases}\E^{-\frac{1}{1-4\xi^2}} &\text{if }|\xi|<\frac{1}{2},\\
0&\text{otherwise}.
\end{cases}$$
is even, of class $\mathcal{C}^\infty$ and with compact support $[-\frac{1}{2},\frac{1}{2}]$. We set $\widehat{\rho_\star}=\upsilon$, so that $$\rho_\star(x)=\frac{1}{2\pi}\int_{-\frac{1}{2}}^{\frac{1}{2}} \upsilon(\xi)\,\E^{-\I x \xi}\, d\xi$$
by the Fourier inversion theorem. By construction, the Fourier transform of $\rho_\star$ has support equal to $[-\frac{1}{2},\frac{1}{2}]$. Set now 
$$\rho(x)=\frac{(\rho_\star(x))^2}{\int_{\R} (\rho_*(y))^2\,dy}.$$ 
By construction, $\rho$ is smooth, even, non-negative and with integral equal to $1$. Moreover, $\widehat{\rho}$ is up to a constant equal to $\nu*\nu(\xi)$, so it has support included into $[-1,1]$. The convolution of $\rho$ with characteristic functions of intervals will allow us to transform estimates on test functions into estimates on cumulative distribution functions. More precisely, for $a \in \R$ and $\eps >0$, set
$\rho_\eps(x)=\frac{1}{\eps}\,\rho(\frac{x}{\eps})$, and $\phi_{a,\eps}(x)=\phi_\eps(x-a)$, where $\phi_\eps$ is the function $1_{(-\infty,0]}*\rho_\eps$. One sees $\phi_{a,\eps}$ as a smooth approximation of the characteristic function $1_{(-\infty,a]}$.\bigskip

For all $a,\eps$, $\phi_{a,\eps}$ has Fourier transform compactly supported on $\left[-\frac{1}{\eps},\frac{1}{\eps}\right]$. Moreover, it has negative derivative, and decreases from $1$ to $0$. Later, we will use the identity
$$\phi_\eps(\eps x)=\phi_1(x)=\phi(x).$$
On the other hand, we have the following estimates for $K>0$ (we used \textsc{Sage} for numerical computations):
\begin{align*}
 |\rho_*(K)| &= \frac{1}{2\pi K^2}\left|\int_{-\frac{1}{2}}^{\frac{1}{2}} \upsilon''(\xi)\,\E^{-\I K \xi}\, d\xi\right| \leq \frac{1}{2\pi K^2} \int_{0}^{\frac{1}{2}} |\upsilon''(\xi)|\,d\xi = \frac{1.0166_-}{K^2};\\
\int_{\R} (\rho_*(y))^2\,dy&=\frac{1}{2\pi} \int_0^{\frac{1}{2}} |\upsilon(\xi)|^2\,d\xi = 0.01059_+.\end{align*}
Therefore, for any $K>0$,
\begin{align*}
\rho(K) &= \rho(-K)=\frac{(\rho_*(K))^2}{\int_{\R} (\rho_*(y))^2\,dy} \leq \frac{99}{K^4};\\
\phi(K)&=1-\phi(-K)=\int_{0}^\infty \rho(K+y)\,dy \leq \frac{33}{K^3}.
\end{align*}
\medskip

\begin{lemma}\label{lem:tao}
Let $V$ and $W$ be two random variables with cumulative distribution functions $F_V$ and $F_W$. Assume that for some $\eps>0$
$$|\esper[\phi_{a,\eps}(V)]-\esper[\phi_{a,\eps}(W)] | \leq B\eps,$$
where the positive constant $B$ is independent of $a$. We also suppose that $W$ has a density w.r.t. Lebesgue measure that is bounded by $m$. Then, 
$$\sup_{a \in \R} |F_V(a)-F_W(a)| \leq 2(B+10m)\,\eps.$$ 
\end{lemma}

\begin{proof}
Fix a positive constant $K$, and denote $\delta=\sup_{a \in \R}|F_V(a)-F_W(a)|$ the Kolmo\-go\-rov distance between $V$ and $W$. One has
\begin{align*}F_V(a) = \esper[1_{V \leq a}] &\leq \esper[\phi_{a+K\eps,\eps}(V)] + \esper[(1-\phi_{a+K\eps,\eps}(V))\,1_{V \leq a}]\\
&\leq \esper[\phi_{a+K\eps,\eps}(W)] + \esper[(1-\phi_{a+K\eps,\eps}(V))\,1_{V \leq a}] + B\eps.
\end{align*}
The second expectation writes as
\begin{align*}
\esper[(1-&\phi_{a+K\eps,\eps}(V))\,1_{V \leq a}]=\int_{\R} (1-\phi_{a+K\eps,\eps}(x))\,1_{(-\infty,a]}(x)\, f_V(x)\,dx\\
&=-\int_{\R} ((1-\phi_{a+K\eps,\eps}(x))\,1_{(-\infty,a]}(x))'\, F_V(x)\,dx\\
&=\int_{\R} \phi_{a+K\eps,\eps}'(x)\,1_{(-\infty,a]}(x)\, F_V(x)\,dx+\int_{\R} (1-\phi_{a+K\eps,\eps}(x))\,1_a(x)\, F_V(x)\,dx.
\end{align*}
For the first integral, since $F_V(x)\geq F_W(x)-\delta$ and the derivative of $\phi_{a+K\eps,\eps}$ is negative, an upper bound on $I_1$ is
\begin{align*}& \int_{\R} \phi_{a+K\eps,\eps}'(x)\,1_{(-\infty,a]}(x)\, F_W(x)\,dx - \delta\int_{\R}\phi_{a+K\eps,\eps}'(x)\,1_{V \leq a}(x) \\
&= \int_{\R} \phi_{a+K\eps,\eps}'(x)\,1_{(-\infty,a]}(x)\, F_W(x)\,dx +  (1-\phi_{a+K\eps,\eps}(a))\,\delta\\
&=\int_{\R} \phi_{a+K\eps,\eps}'(x)\,1_{(-\infty,a]}(x)\, F_W(x)\,dx + (1-\phi(-K))\,\delta.
\end{align*}
As for the second integral, it is simply $(1-\phi_{a+K_\eps,\eps}(a)) F_V(a)$, and by writing $F_V(a) \leq F_W(a)+\delta$, one gets the upper bound on $I_2$
\begin{align*}
&\int_{\R} (1-\phi_{a+K\eps,\eps}(x))\,1_a(x)\, F_W(x)\,dx + (1-\phi_{a+K\eps,\eps}(a))\,\delta \\
&=\int_{\R} (1-\phi_{a+K\eps,\eps}(x))\,1_a(x)\, F_W(x)\,dx + (1-\phi(-K))\,\delta .
\end{align*}
One concludes that 
$$\esper[(1-\phi_{a+K\eps,\eps}(V))\,1_{V \leq a}]\leq \esper[(1-\phi_{a+K\eps,\eps}(W))\,1_{W \leq a}]+2 (1-\phi(-K))\delta.$$
On the other hand, if $m$ is a bound on the density $f_W$ of $W$, then
\begin{align*}\esper[\phi_{a+K\eps,\eps}(W)\,1_{W \geq a}] &= \int_a^{\infty} \phi_{a+K\eps,\eps}(y)\,f_W(y)\,dy\\
&\leq m\int_{a}^\infty \phi_{\eps}(y-a-K\eps)\,dy = m\int_0^\infty \phi_\eps(y-K\eps)\,dy\\
&\leq m\eps \int_0^\infty \phi(u-K)\,du \leq m \eps \left(K+4.82\right),
\end{align*}
by using on the last line the bound $\phi(x) \leq \frac{33}{x^3}$. As a consequence,
\begin{align*}
\esper[\phi_{a+K\eps,\eps}(W)] &\leq \esper[\phi_{a+K\eps,\eps}(W)\,1_{W \leq a}] + m\left(K+4.82\right)\eps\\
F_V(a)&\leq F_W(a) + \left(B+m\left(K+4.82\right)\right) \eps + 2\,\frac{33}{K^3}\,\delta.
\end{align*}
Similarly, $F_V(a)\geq F_W(a) - (B+m(K+4.82)) \eps - 2\,\frac{33}{K^3}\,\delta$, so in the end
$$\delta = \sup_{a \in \R}|F_V(a)-F_W(a)| \leq \left(B+m\left(K+4.82\right)\right)\eps + \frac{66}{K^3}\,\delta.$$\bigskip
As this is true for every $K$, one can for instance take $K=\sqrt[3]{132}$, which gives 
$$\delta \leq \frac{1}{1-\frac{1}{2}}\left(B+m\left(\sqrt[3]{132}+4.82\right)\right)\eps \leq  2(B + 10m)\,\eps.$$
\end{proof}
\bigskip

We are going to apply Lemma \ref{lem:tao} with $V=\frac{Y_n}{n^{1/2}}$ and $W=W_n$. First, notice that a bound on the density of $W_n$ is
$$\frac{|\psi_n(x)|}{I_n}\leq \frac{1}{I_\infty}=\frac{2}{12^{1/4}\,\Gamma(\frac{1}{4})}=m.$$
On the other hand, using the Fourier transform of the Heaviside function
$$
\widehat{1}_{(-\infty,a]}(\xi)=\E^{\I a \xi}\,\left(\pi\delta_0(\xi)+\frac{\I}{\xi}\right),
$$
we get
\begin{align*}
\esper\left[\phi_{a,\eps}\!\left(\frac{Y_n}{n^{1/2}}\right)\right]-\esper[\phi_{a,\eps}(W_n)] &= \frac{1}{2\pi I_n} \int_{-\frac{1}{\eps}}^{\frac{1}{\eps}} \widehat{\phi_{a,\eps}}(\xi)\,\widehat{\psi_n}(\xi) \left(\E^{\frac{\xi^2}{2n^{1/2}}}-1\right)\,d\xi\\
&=\frac{1}{2\pi I_n} \int_{-\frac{1}{\eps}}^{\frac{1}{\eps}} \widehat{\rho_{\eps}}(\xi)\,\E^{\I a\xi}\,\left(\frac{\I}{\xi}\right)\,\widehat{\psi_n}(\xi) \left(\E^{\frac{\xi^2}{2n^{1/2}}}-1\right)\,d\xi ;\\
\left| \esper\!\left[\phi_{a,\eps}\!\left(\frac{Y_n}{n^{1/2}}\right)\right]-\esper[\phi_{a,\eps}(W_n)]\right| &\leq \frac{1}{2\pi I_n\,n^{1/2}} \int_{-\frac{1}{\eps}}^{\frac{1}{\eps}} |\widehat{\rho}(\eps\xi)|\, |\widehat{\psi_n}(\xi)|\,\E^{\frac{\xi^2}{2n^{1/2}}}\,d\xi
\end{align*}
by controlling $\E^{\frac{\xi^2}{2n^{1/2}}}-1$ by its first derivative (notice that we used the vanishing of this quantity at $\xi=0$ in order to compensate the singularity of the Fourier transform of the Heavyside distribution). Since $\|\widehat{\rho}\|_{\mathrm{L}^\infty}=\|\rho\|_{\mathrm{L}^1}=1$, the previous bound can be rewritten as
$$
\left| \esper\!\left[\phi_{a,\eps}\!\left(\frac{Y_n}{n^{1/2}}\right)\right]-\esper[\phi_{a,\eps}(W_n)]\right| \leq \frac{1}{2\pi I_n\,n^{1/2}} \int_{-\frac{1}{\eps}}^{\frac{1}{\eps}} |\widehat{\psi_n}(\xi)|\,\E^{\frac{\xi^2}{2n^{1/2}}}\,d\xi.
$$\bigskip

We then need estimates on the Fourier transform of $\widehat{\psi_n}$, and more precisely estimates of exponential decay. To this purpose, we use the following Lemma, which is related to \cite[Theorem IX.13, p. 18]{RS75}:

\begin{lemma}
Let $f$ be a function which is analytic on a band $\{z \in \CC\,\,|\,\,|\mathrm{Im}(z)|<c\}.$ 
For any $b \in (0,c)$,
$$|\widehat{f}(\xi) |\leq 2\left(\sup_{-b\leq a \leq b} \|f(\cdot + \I a)\|_{\leb^1}\right)\,\E^{-b|\xi|},$$
assuming that the supremum is finite.
\end{lemma}

\begin{proof}
Notice that the Fourier transform of $\tau_af(\cdot) =f(\cdot + \I a)$ is
$$\int_{\R} \tau_a f(x)\,\E^{\I x \xi}\,dx = \int_{\R} f(x+\I a)\,\E^{\I x \xi}\,dx = \left(\int_{\R} f(x+\I a)\,\E^{\I (x-\I a) \xi}\,dx\right)\E^{-a\xi}.$$
By analyticity of the function in the integral, using Cauchy's integral formula, one sees that the last term is also
$$\left(\int_{\R} f(x)\,\E^{\I x \xi}\,dx\right)\E^{a\xi} = \widehat{f}(\xi)\,\E^{-a\xi},$$
(see the details on page 132 of the book by Reed and Simon). It follows that
$$|\widehat{f}(\xi)|\,\E^{a|\xi|} \leq |\widehat{f}(\xi)|\,(\E^{a\xi}+\E^{-a\xi}) \leq |\widehat{\tau_af}(\xi)|+|\widehat{\tau_{-a}f}(\xi)| \leq \|\tau_af\|_{\leb^1}+\|\tau_{-a}f\|_{\leb^1}.$$
\end{proof}
\bigskip

Thus we need to compute for $a>0$ the $\leb^1$-norm of $\psi_n(\cdot+\I a)$. We write
\begin{align*}
|\psi_n(x+\I a)| &= \E^{-\frac{n^{1/2}(x^2-a^2)}{2}}\,\left|\cosh\left(\frac{x+\I a}{n^{1/4}}\right)\right|^n\\
&=|\psi_n(x)|\, \E^{\frac{n^{1/2}a^2}{2}}\, \left|\cos^2\left(\frac{a}{n^{1/4}}\right)+ \tanh^2\left(\frac{x}{n^{1/4}}\right) \,\sin^2\left(\frac{a}{n^{1/4}}\right)\right|^{\frac{n}{2}}\\
&=|\psi_n(x)|\, \E^{\frac{n^{1/2}a^2}{2}}\, \left|1-\left(1- \tanh^2\left(\frac{x}{n^{1/4}}\right) \right)\,\sin^2\left(\frac{a}{n^{1/4}}\right)\right|^{\frac{n}{2}}.
\end{align*}
For $n$ large enough, $\sin^2(\frac{a}{n^{1/4}}) \geq \frac{a^2}{n^{1/2}}-\frac{a^4}{3n}$, and on the other hand, $0\leq \tanh^2\left(\frac{x}{n^{1/4}}\right) \leq \frac{x^2}{n^{1/2}}$, so
\begin{align*}
\left|\frac{\psi_n(x+\I a)}{\psi_n(x)}\right| &\leq \E^{\frac{n^{1/2}a^2}{2}} \exp\left(-\frac{n^{1/2}a^2}{2}\left(1- \tanh^2\left(\frac{x}{n^{1/4}}\right)\right) \left(1-\frac{a^2}{3n^{1/2}}\right)\right)\leq \E^{\frac{a^4}{3}}\,\E^{\frac{a^2x^2}{2}}.
\end{align*}
Since $\psi_n(x)$ behaves as $\E^{-x^4/12}$, the previous Lemma can be applied, with an asymptotic bound
\begin{align*}
\|\psi_n(\cdot + \I a)\|_{\leb^1} &\lesssim \E^{\frac{a^4}{3}}\,\int_{\R} \E^{-\frac{x^4}{12}+\frac{a^2x^2}{2}}\,dx = \E^{\frac{13a^4}{12}} \int_{\R} \E^{-\frac{(x^2-3a^2)^2}{12}}\,dx \\
&\lesssim  \E^{\frac{13a^4}{12}}\left(2\sqrt{3}\,a+I_\infty\right)
\end{align*}
by cutting the integral in two parts according to the sign of $x^2-3a^2$. We have therefore proven:

\begin{proposition}\label{prop:sizefourier}
For any $b \geq 0$, 
$$|\widehat{\psi}_n(\xi)| \lesssim K(b)\,\E^{-b|\xi|},$$
where $K(b) = 2\E^{\frac{13b^4}{12}}(2\sqrt{3}\,b+I_\infty)$ and where the symbol $\lesssim$ means that the inequality is true up to any multiplicative constant $1+\eps$, for $\eps>0$ and $n$ large enough.
\end{proposition}
\bigskip

We can now conclude. Fix $b>0$, and $D<2b$. On the interval $[-Dn^{1/2},Dn^{1/2}]$,  we have
$$\frac{\xi^2}{2n^{1/2}}-b|\xi| = -|\xi|\left(b-\frac{|\xi|}{2n^{1/2}}\right) \leq - |\xi|\left(b-\frac{D}{2}\right).$$
Therefore, with $\eps=\frac{1}{Dn^{1/2}}$,
\begin{align*}\left| \esper\!\left[\phi_{a,\eps}\!\left(\frac{Y_n}{n^{1/2}}\right)\right]-\esper[\phi_{a,\eps}(W_n)]\right| &\lesssim \frac{K(b)}{2\pi I_\infty\,n^{1/2}} \int_{\R} \E^{-\left(b-\frac{D}{2}\right)|\xi|}\,d\xi= \frac{K(b)}{\pi I_\infty\,n^{1/2}\,\left(b-\frac{D}{2}\right)}\\
&\lesssim \frac{K(b)\, D}{\pi I_\infty \,\left(b-\frac{D}{2}\right)}\,\eps.
\end{align*}
So, Lemma \ref{lem:tao} applies to $V=\frac{Y_n}{n^{1/2}}$ and $W=W_n$, with
$$d_{\mathrm{Kol}}\left(\frac{Y_n}{n^{1/2}},W_n\right)\lesssim 2\left(\frac{K(b)\,D}{\pi I_\infty \,\left(b-\frac{D}{2}\right)}+\frac{10}{I_\infty}\right)\,\eps=\frac{2}{I_\infty\,n^{1/2}}\left(\frac{K(b)}{\pi\,(b-\frac{D}{2})}+\frac{10}{D}\right).$$
Taking $b=D=0.77$, we get finally
\begin{align*}
d_{\mathrm{Kol}}\left(\frac{Y_n}{n^{1/2}},W_n\right)\lesssim \frac{2}{I_\infty\,n^{1/2}} \left( \frac{2\,K(b)}{\pi \,b}+\frac{10}{b}\right)\leq \frac{10.27}{n^{1/2}}.
\end{align*}
Adding the bound on $d_{\mathrm{Kol}}(W_n,W_\infty)$ yields then:

\begin{theorem}
For $n$ large enough, $$d_{\mathrm{Kol}}\left(\frac{Y_n}{n^{1/2}},W_\infty\right)\leq 11\,n^{-1/2}.$$
\end{theorem}
Notice that we have only used arguments of Fourier analysis and the language of mod-Gaussian convergence in order to get this bound.
\medskip

\subsection{Local limit theorem}
Combining Proposition \ref{prop:sizefourier} with Theorem 5 in \cite{DKN11} on local limit theorems for mod-$\phi$ convergence, we obtain the following local limit theorem for the magnetization in the Curie-Weiss model:

\begin{theorem}
In the Curie-Weiss model, if we note $M_n$ for the total magnetization, then we have:
$$\lim_{n\to\infty} n^{1/2}\, \proba[n^{-1/4} M_n\in B]=\frac{2}{12^{1/4} \Gamma(\frac{1}{4})}\, m(B),$$
for relatively compact sets $B$ with $m(\partial B)=0$, $m$ denoting the Lebesgue measure.
\end{theorem}

\begin{proof}
With the notation of \S\ref{subsec:speedofconv}, $Y_n=n^{-1/4} M_n$ and we need to check assumptions \textbf{H1}, \textbf{H2} and \textbf{H3} of \cite{DKN11} for $(Y_n)_{n\in \N}$ in order to apply Theorem 5 in \emph{loc. cit.} 
\begin{itemize}
	\item[$\bullet$ \textbf{H1.}] The Fourier transform of the limit law $\mu(dx)=\frac{\psi(x)\,dx}{I_\infty}$ of $\frac{Y_n}{n^{1/2}}$ is in the Schwartz space, hence is integrable.\vspace{2mm}
	\item[$\bullet$ \textbf{H2.}] The Fourier transforms $\frac{\widehat{\psi_n}(\xi)}{I_n}\,\E^{\frac{\xi^2}{2n^{1/2}}}$ of $\frac{Y_n}{n^{1/2}}$ converge locally uniformly in $\xi$ towards the Fourier transform $\frac{\widehat{\psi}(\xi)}{I_\infty}$. Indeed, by Theorem \ref{thm:ellnew}, $$\frac{\psi_n(x)}{I_n}\to_{\mathrm{L}^1(\R)}\frac{\psi(x)}{I_\infty},$$
	 so $\frac{\widehat{\psi_n}(\xi)}{I_n} \to \frac{\widehat{\psi}(\xi)}{I_\infty}$, and the term $\E^{\frac{\xi^2}{2n^{1/2}}}$ converges locally uniformly to $1$.\vspace{2mm}
	\item[$\bullet$ \textbf{H3.}] Finally, we have to prove that for all $k\geq0$, 
	$$f_{n,k}(\xi)=\esper\left[\E^{\I \xi \frac{Y_n}{n^{1/2}}}\right] \,1_{|\xi|\leq k n^{1/2}}$$ is uniformly integrable. Following Remark 2 in \cite{DKN11}, it is enough to show that 
	$$\left|\esper\left[\E^{\I \xi \frac{Y_n}{n^{1/2}}}\right]\right|\leq h(\xi)$$
	 for $\xi$ such that $|\xi|\leq k n^{1/2}$ for some non-negative and integrable function $h$ on $\R$. This is a consequence of Proposition \ref{prop:sizefourier}: since $|\widehat{\psi_n}(\xi)|\leq C(k)\,\E^{-k|\xi|}$ for any $k>0$, one can write
	 \begin{align*}
	 \left|\esper\left[\E^{\I \xi \frac{Y_n}{n^{1/2}}}\right]\right|&=\frac{|\widehat{\psi_n}(\xi)|}{I_n}\,\E^{\frac{\xi^2}{2n^{1/2}}} \\
	 &\leq \frac{C(k)}{I_\infty} \,\E^{-k|\xi|+\frac{\xi^2}{2n^{1/2}}} \\
	 &\leq \frac{C(k)}{I_\infty} \,\E^{-\frac{k}{2}\,|\xi|}
	 \end{align*}
	 for any $|\xi|<k\,n^{1/2}$. We can hence apply Theorem 5 of \cite{DKN11} with $\frac{d\mu}{dm}(0)=1/I_{\infty}$,and the value of $I_{\infty}$ was computed in the proof of Lemma \ref{lem:tao}.
\end{itemize}
\end{proof}

\begin{remark}
A similar result would more generally hold for $Y_n$ whenever one has some estimates of exponential decay on $\widehat{\psi_n}(\xi)$ similar to the one given in Lemma \ref{prop:sizefourier}:
$$\lim_{n\to\infty} t_n \,\proba[Y_n\in B]=\frac{1}{I_\infty}\,  m(B).$$
In particular, the result holds for the random walks changed in measure studied in \S\ref{subsec:randomwalkchange}.
\end{remark}

\begin{remark}
The idea behind the proof the local limit theorem above and which is found in \cite{DKN11} is the following: thanks to approximation arguments, one can show that it is enough to prove the local limit theorem for functions whose Fourier transforms have compact support (instead of indicator functions $1_B$). Then, one uses Parseval's relation for such functions $f$ to write:
$$\esper[f(Y_n)]=\frac{1}{2\pi}\int_\R \frac{\widehat{\psi_n}(\xi)}{I_n}\, \E^{\frac{\xi^2}{2t_n}} \widehat{f}\left(-\frac{\xi}{t_n}\right)\, d\xi $$
and then use the assumptions to conclude.
\end{remark}
\bigskip
\bigskip

\section{Mod-Gaussian convergence for the Ising model: the cumulant method}\label{sec:cumulantspin}
In this appendix, we give another combinatorial proof of the mod-Gaussian convergence of the magnetization in the Ising model, without ever computing the Laplace transform of $M_n$. This serves as an illustration of the cumulant method developed in \cite{FMN13}.
\medskip

\subsection{Joint cumulants of the spins} 
When $\alpha=0$, one can realize the Ising model by choosing $\sigma(1)$ according to a Bernoulli random variable of parameter $\frac{1}{2}$, and then each sign $X_i=\sigma(i)\sigma(i+1)$ according to independent Bernoulli random variables with $$\proba[X_i=1]=1-\proba[X_i=-1]=\frac{\E^{\beta}}{2\cosh \beta}.$$ 
In particular, one recovers immediately the value of the partition function $Z_n(\IS,0,\beta)=2^n(\cosh \beta)^{n-1}$. We then want to compute the joint cumulants of the magnetization $M_n$; by parity, the odd cumulants and moments vanish. By multilinearity, one can expand
$$\kappa^{(2r)}(M_n)=\sum_{i_1,\ldots,i_{2r}=1}^n\kappa(\sigma(i_1),\ldots,\sigma(i_{2r})),$$
so the problem reduces to the computation of the joint cumulants of the individual spins, and to the gathering of these quantities. Notice that the joint moments of the spins can be computed easily. Indeed, fix $i_1\leq i_2 \leq \cdots \leq i_{2r}$, and let us calculate $\esper[\sigma(i_1)\cdots\sigma(i_{2r})]$. If $i_{2r-1}=i_{2r}$, then the two last terms cancel and one is reduced to the computation of a joint moment of smaller order. Otherwise, notice that
\begin{align*}\esper[\sigma(i_1)\cdots\sigma(i_{2r-2})\sigma(i_{2r-1})\sigma(i_{2r})]&=\esper[\sigma(i_1)\cdots\sigma(i_{2r-2})X_{i_{2r-1}}X_{i_{2r-1}+1}\cdots X_{i_{2r}-1}]\\
&=\esper[\sigma(i_1)\cdots\sigma(i_{2r-2})]\,x^{i_{2r}-i_{2r-1}} \quad\text{where }x=\tanh \beta.
\end{align*}
By induction, we thus get
$$\esper[\sigma(i_1)\cdots\sigma(i_{2r})]=x^{(i_2-i_1)+(i_4-i_3)+\cdots+(i_{2r}-i_{2r-1})}.$$\bigskip

Let us then go to the joint cumulants. We fix $i_1 \leq i_2 \leq \cdots \leq i_{2r}$, and to simplify a bit the notations, we denote $i_1=\mathbf{1}$, $i_2=\mathbf{2}$, \emph{etc}. We recall that the joint cumulants write as
$$\kappa(\sigma(\mathbf{1}),\ldots,\sigma(\mathbf{2r}))=\sum_{\Pi \in \qym_{2r}}\mu(\Pi)\,\prod_{A \in \Pi}\,\esper\!\left[\prod_{a \in A} \sigma(\mathbf{a})\right],$$
where the sum runs over set partitions of $\lle 1,2r\rre$. By parity, the set partitions with odd parts do not contribute to the sum, so one can restrict oneself to the set $\qym_{2r,\text{even}}$ of even set partitions. If $A=\{a_1<\cdots<a_{2s}\}$ is an even part of $\lle 1,2r\rre$, we write $x^{p(A)}=x^{(\mathbf{a}_{2}-\mathbf{a}_{1})+\cdots+(\mathbf{a}_{2s}-\mathbf{a}_{2s-1})}$. Thus,
$$\kappa(\sigma(\mathbf{1}),\ldots,\sigma(\mathbf{2r}))=\sum_{\Pi \in \qym_{2r,\text{even}}}\mu(\Pi)\prod_{A \in \Pi} x^{p(A)}.$$
In this polynomial in $x$, several set partitions give the same power of $x$; for instance, with $2r=4$, the set partitions $\{1,2,3,4\}$ and $\{1,2\} \sqcup \{3,4\}$ both give $x^{(\mathbf{2}-\mathbf{1})+(\mathbf{3}+\mathbf{4})}$. Denote $\pym_{2r}$ the set of set partitions of $\lle 1,2r\rre$ whose parts are all of cardinality $2$ (pair set partitions, or pairings). To every even set partition $\Pi$, one can associate a pairing $p(\Pi)$ by cutting all the even parts $\{a_1<a_2<\cdots<a_{2s-1}<a_{2s}\}$ into the pairs $\{a_1<a_2\},\ldots,\{a_{2s-1}<a_{2s}\}$. For instance, the even set partition $\Pi=\{1,3,4,5\}\sqcup\{2,6\}$ gives the pairing $(1,3)(4,5)(2,6)$. Then, with obvious notations,
\begin{equation}\kappa(\sigma(\mathbf{1}),\ldots,\sigma(\mathbf{2r}))=\sum_{\Pi \in \qym_{2r,\text{even}}}\mu(\Pi)\, x^{p(\Pi)}.
\label{eq:jointcumulant1}
\end{equation}
In Equation \eqref{eq:jointcumulant1}, two important simplifications can be made:\vspace{2mm}
\begin{enumerate}
\item One can gather the even set partitions $\Pi$ according to the pairing $\rho=p(\Pi) \in \pym_{2r}$ that they produce. It turns out that the corresponding sum of M\"obius functions $F(\rho)$ has a simple expression in terms of the pairing, see \S\ref{subsec:functionalF}.  \vspace{1mm}
\item Some pairings $\rho$ yield the same monomial $x^{\rho}$ and the same functional $F(\rho)$. By gathering these contributions, one can reduce further the complexity of the sum, see \S\ref{subsec:functionalN}.\vspace{1mm}
\end{enumerate}
In the end, we shall obtain an exact formula for $\kappa(\sigma(\mathbf{1}),\ldots,\sigma(\mathbf{2r}))$ that writes as a sum over Dyck paths of length $2r-2$, with simple coefficients; see Theorem \ref{thm:jointcumulant}.

\subsubsection{Pairings, labelled Dyck paths and labelled planar trees}\label{subsec:combinatorics}
Before we start the reduction of Formula \eqref{eq:jointcumulant1}, it is convenient to recall some facts about the combinatorial class of pairings. We have defined a pairing $\rho$ of size $2r$ to be a set partition of $\lle 1,2r\rre $ in $r$ pairs $(a_1,b_1),\ldots,(a_r,b_r)$. There are 
$$\card\, \pym_{2r}=(2r-1)!!=(2r-1)(2r-3)\cdots 3\,1$$ pairings of size $2r$, and it is convenient to represent them by diagrams:
\figcap{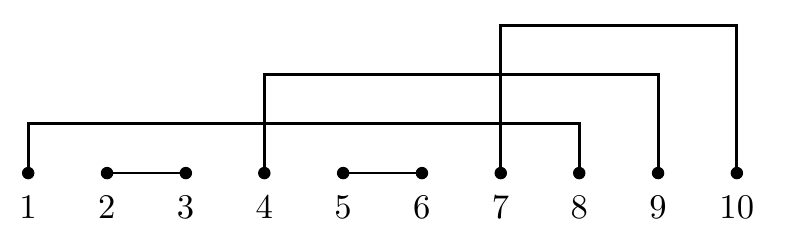}{The diagram of a pairing of size $2r=10$.\label{fig:pairing}}\vspace{-5mm}
\figcap{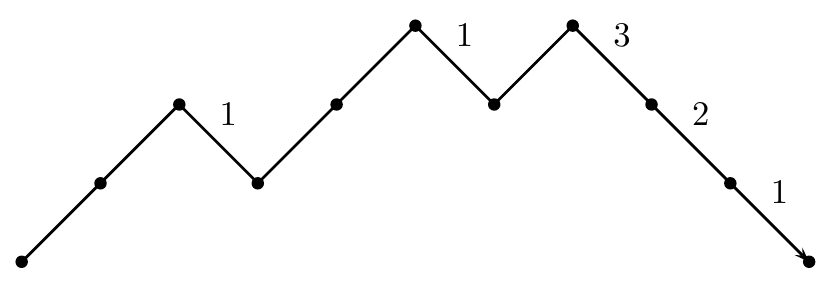}{The labelled Dyck path corresponding to the pairing of Figure \ref{fig:pairing}.\label{fig:dyckpath}}
On the other hand, a labelled Dyck path of size $2r$ is a path $\delta : \lle 0,2r\rre \to \N$ with $2r$ steps either ascending or descending, such that:\vspace{2mm}
\begin{itemize}
\item the path $\delta$ starts from $0$, ends at $0$ and stays non-negative; \vspace{1mm}
\item each descending step $\delta(k)>\delta(k+1)$ is labelled by an integer $i \in \lle 1,\delta(k)\rre$.\vspace{1mm}
\end{itemize}
From a labelled Dyck path of size $2r$, one constructs a pairing on $2r$ points as follows: one reads the diagram from left to right, opening a bond when the path is ascending, and closing the $i$-th opened bond available from right to left when the path is descending with label $i$. For instance, if one starts from the Dyck path of Figure \ref{fig:dyckpath}, one obtains the pairing of Figure \ref{fig:pairing}. This provides a first bijection between pairings $\rho$ and labelled Dyck paths $\delta$. \bigskip

By considering a Dyck path as the code of the depth-first traversal of a rooted tree, one obtains a second bijection betwen pairings of size $2r$ and labelled planar rooted trees with $r$ edges. Here, by labelled planar rooted tree, we mean a planar rooted tree with a label $i$ on each edge $e$ that is between $1$ and the height $h(e)$ of the edge (with respect to the root). For instance, the following labelled tree $T$ corresponds to the Dyck path of Figure \ref{fig:dyckpath} and to the pairing of Figure \ref{fig:pairing}:

\figcap{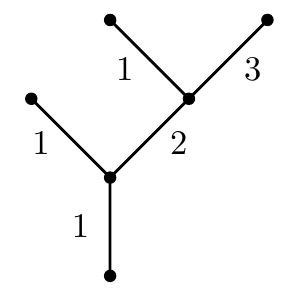}{The labelled planar rooted tree corresponding to the pairing of Figure \ref{fig:pairing}.}

We shall denote $\tym_r$ the set of planar rooted trees with $r$ edges (without label), and $\dym_{2r}$ the corresponding set of Dyck paths (again without label); they have cardinality
$$\card \,\tym_r = \card \,\dym_{2r}= C_r = \frac{1}{r+1}\binom{2r}{r}.$$
They correspond to the subset $\nym_{2r}$ of $\pym_{2r}$ that consists in non-crossing pair partitions of $\lle 1,2r\rre$; a bijection is obtained by labelling each edge or descending step by $1$, and by using the previous constructions. For instance, the non-crossing pairing, the Dyck path and the planar rooted tree of Figure \ref{fig:restriction} do correspond. 
\figcap{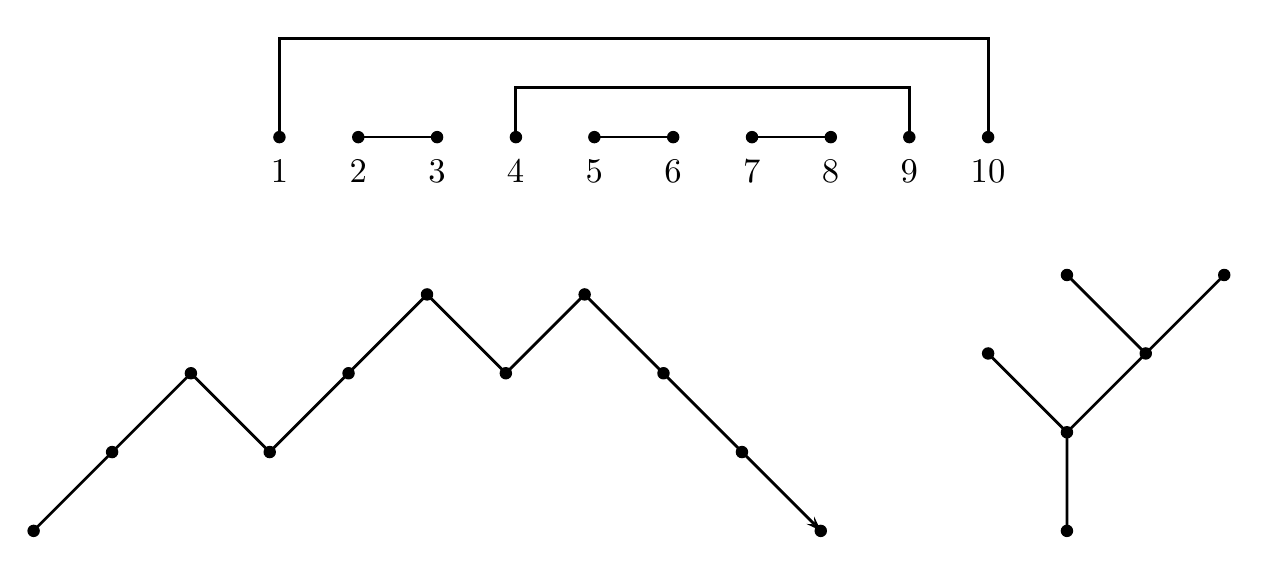}{Bijection between non-crossing pairings, Dyck paths and planar rooted trees.\label{fig:restriction}}
\bigskip

In what follows, we shall always use the letters $\nu$, $\delta$ and $T$ respectively for non-crossing pairings, for Dyck paths and for planar rooted trees. We shall then use constantly the bijections described above, and denote for instance $\nu(T)$ for the non-cros\-sing pairing associated to a tree $T$, or $\delta(\nu)$ for the Dyck path associated to a non-cros\-sing pairing $\nu$. We shall also use the exponent $+$ to indicate the following operations on these combinatorial objects:\vspace{1mm}
\begin{itemize}
\item transforming a non-crossing pairing $\nu$ of size $2r-2$ in a non-crossing pairing $\nu^+$ of size $2r$ by adding the bond $\{1,2r\}$ "over" the bonds of $\nu$.\vspace{1mm}
\item transforming a Dyck path $\delta$ of length $2r-2$ in a Dyck path $\delta^+$ of length $2r$ by adding an ascending step before $\delta$ and a descending step after $\delta$.\vspace{1mm}
\item transforming a rooted tree $T$ with $r-1$ edges in a rooted tree $T^+$ with $r$ edges by adding an edge "below" the root.\vspace{1mm}
\end{itemize}
All these operations are compatible with the aforementioned bijections, so for instance $\nu(T^+)=(\nu(T))^+$ and $\delta(\nu^+)=(\delta(\nu))^+$.

\subsubsection{Uncrossing pairings and the associated poset}\label{subsec:functionalN}
 Let us now see how the combinatorics of pairings, Dyck paths and planar rooted trees intervene in Formula \eqref{eq:jointcumulant1}. We start by gathering the set partitions $\Pi$ with the same associated pairing $\rho=p(\Pi)$. Thus, let us write
$$\kappa(\sigma(\mathbf{1}),\ldots,\sigma(\mathbf{2r}))=\sum_{\rho \in \pym_{2r}} x^{\rho}\,\left(\sum_{\substack{\Pi \in \qym_{2r,\text{even}}\\ p(\Pi)=\rho}} \mu(\Pi) \right)=\sum_{\rho \in \pym_{2r}} x^\rho\,F(\rho),$$
where $F(\rho)$ stands for the sum in parentheses. Notice that $x^\rho$ is invariant if one replaces in a pairing two crossing pairs $\{a_1,a_3\},\{a_2,a_4\}$ with $a_1<a_2<a_3<a_4$ by two nested pairs (but non-crossing) $\{a_1,a_4\},\{a_2,a_3\}$; indeed,
$$(a_3-a_1)+(a_4-a_2)=(a_4-a_1)+(a_3-a_1).$$
We call uncrossing the operation on pairings which consists in replacing two crossing pairs by two nested pairs as described above, and we denote $\rho_1 \succeq \rho_2$ if there is a sequence of uncrossings from the pairing $\rho_1$ to the pairing $\rho_2$; this is a partial order on the set $\pym_{2r}$.
\figcap{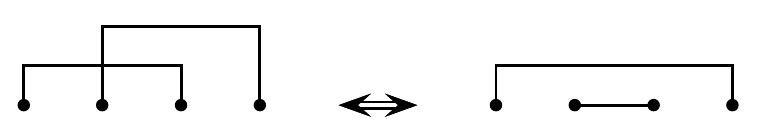}{The operation of uncrossing on a pairing.}

\begin{proposition}
The poset $(\pym_{2r},\preceq)$ is a disjoint union of lattices, and each lattice contains a unique non-crossing set partition $\nu$, which is the minimum of this connected component of the Hasse diagram of $(\pym_{2r},\preceq)$. Moreover:\vspace{2mm}
\begin{enumerate}[1.]
\item On the lattice $L(\nu)$ associated to $\nu \in \nym_{2r}$, the monomial $x^{\rho}$ and the functional $F(\rho)$ are constant (equal to $x^{\nu}$ and $F(\nu)$). \vspace{1mm}
\item The cardinality $\card\, L(\nu)=N(\nu)$ is given by:
$$N(\nu)=\prod_{e \in E(T(\nu))} h(e,T(\nu)),$$
where $h(e,T)$ is the height of the edge $e$ in the (planar) rooted tree $T$, and $E(T)$ is the set of edges of a tree $T$.
\end{enumerate}
\end{proposition}

\begin{proof}
First, notice that if $\rho_1 \preceq \rho_2$ in $\pym_{2r}$, then there is a sequence of pairings going from $\rho_1$ to $\rho_2$ such that every two consecutive terms $\mu$ and $\rho$ of the sequence differ only by the replacement of a simple nesting by a simple crossing. By that we mean that we do not need to do replacements such as the one on Figure \ref{fig:complexnesting}, which creates $3$ crossings at once.
\figcap{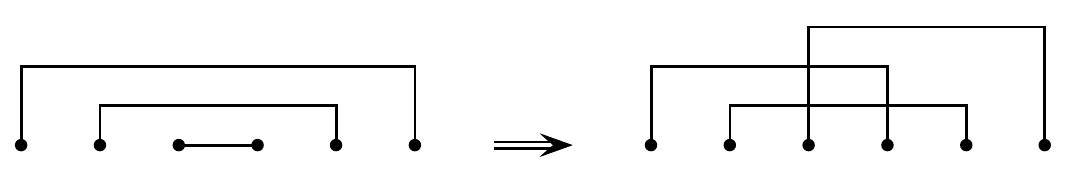}{The crossing of a nesting that is not simple.\label{fig:complexnesting}}
 Indeed, denoting $(i,j)$ the crossing of the $i$-th bond with the $j$-th bond, bonds being numeroted from their starting point, one has $(1,3)=(1,2)\circ (2,3)\circ(1,2)$, which is a composition of simple operations of crossing; and the same idea works for nestings of higher depth. Thus, the Hasse diagram of the poset $(\pym_{2r},\preceq)$ has edges that consist in replacements of simple nestings by simple crossings.\bigskip

This being clarified, it suffices now to notice that \emph{via} the bijection between pairings and labelled Dyck paths explained in \S\ref{subsec:combinatorics}, the replacing a simple nesting by a simple crossing corresponds to the raising of a label by $1$:
\figcap{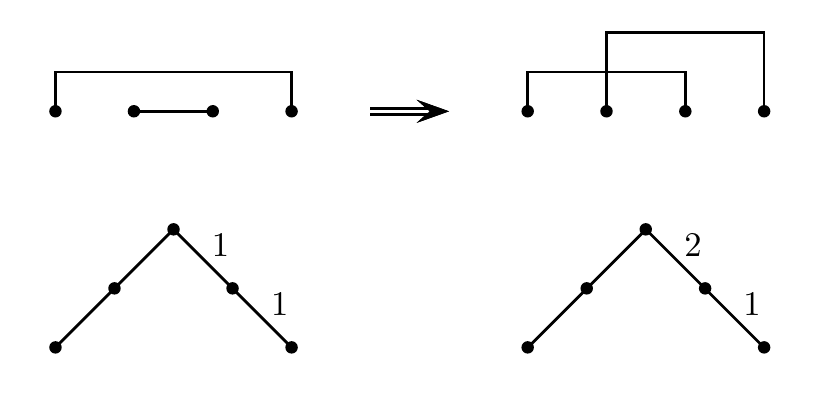}{The operation of uncrossing is a change of labels on Dyck paths.}
In particular, if $\rho_1$ and $\rho_2$ are two comparable pairings in $(\pym_{2r},\preceq)$, then the corresponding labelled Dyck paths have the same shape; and for a given shape $\delta \in \dym_{2r}$, there is exactly one corresponding non-crossing pair partition $\nu=\nu(\delta)$, which is minimal in its connected component in the Hasse diagram of $(\pym_{2r},\preceq)$. Endowed with $\preceq$, this connected component $L(\nu)$ is isomorphic as a poset to the product of intervals
$$\prod_{e \in T(\nu)} \lle 1,h(e,T(\nu))\rre.$$
Indeed, the order on the set of labelled trees of shape $T(\nu)$ induced by $(L(\nu),\preceq)$ and by the bijection between pairings and labelled trees is simply the product of the orders of the intervals of labels. This proves all of the Proposition but the invariance of $F(\cdot)$ on $L(\nu)$ (the invariance of $x^{(\cdot)}$ was shown at the beginning of this paragraph); we devote \S\ref{subsec:functionalF} to this last point and to the actual computation of the functional $F(\cdot)$.
\end{proof}
\bigskip

Assuming the invariance of $F(\cdot)$ on each lattice $L(\nu)$, we thus get:
\begin{equation}\kappa(\sigma(\mathbf{1}),\ldots,\sigma(\mathbf{2r}))=\sum_{\rho \in \pym_{2r}} x^{\rho}\,F(\rho) = \sum_{\nu \in \nym_{2r}} x^{\nu}\,N(\nu)\,F(\nu) \label{eq:jointcumulant2},
\end{equation}
where $N(\nu)$ is explicit. Hence, it remains to compute the functional $F(\rho)$.

\subsubsection{Computation of the functional $F$}\label{subsec:functionalF}
The main result of this paragraph is:
\begin{proposition}\label{prop:functionalF}
The functional $F(\cdot)$ is constant on $L(\nu)$, and if $\nu$ is a non-crossing pairing, then
$$F(\nu)=(-1)^{r-1}\,\prod_{\substack{e \in T(\nu)\\h(e,T(\nu))\neq 1}} (h(e,T(\nu))-1)$$
if $T(\nu)$ has a single edge of height $1$, and $0$ otherwise.
\end{proposition}

\begin{lemma}
The functional $F$ vanishes on pairings associated to labelled rooted trees with more than one edge of height $1$.    
\end{lemma}

\begin{proof}
Suppose that $\Pi$ is an even set partition with $p(\Pi)=\rho$; $\rho$ being a pairing of size $2r$ associated to a labelled Dyck path that reaches $0$ after $2a$ steps, with $2r=2a+2b$, $a>0$ and $b>0$ (this is equivalent to the statement "having more than one edge of height $1$"). We denote $\rho_1$ and $\rho_2$ the pairings associated to the two parts of the Dyck path. There are several possibilities:\vspace{1mm}
\begin{itemize}
\item either $\Pi$ can be split as two even set partitions $\Pi_1$ of $\lle 1,2a\rre$ and $\Pi_2$ of $\lle 2a+1,2r\rre$, with respectively $k$ and $l$ parts, and with $p(\Pi_1)=\rho_1$ and $p(\Pi_2)=\rho_2$;\vspace{1mm}
\item or, $\Pi$ is one of the $k\times l$ possible ways to unite two such even set partitions $\Pi_1$ and $\Pi_2$ by joining one part of $\Pi_1$ with one part of $\Pi_2$;\vspace{1mm}
\item or, $\Pi$ is one of the $\binom{k}{2}\times \binom{l}{2} \times 2!$ possible ways to unite two such even set partitions $\Pi_1$ and $\Pi_2$ by joining two parts of $\Pi_1$ with two parts of $\Pi_2$;\vspace{1mm}
\item or, $\Pi$ is one of the $\binom{k}{3}\times \binom{l}{3} \times 3!$ possible ways to unite two such even set partitions $\Pi_1$ and $\Pi_2$ by joining three parts of $\Pi_1$ with three parts of $\Pi_2$;\vspace{1mm}
\item \emph{etc.}\vspace{1mm}
\end{itemize}
So, $F(\rho)$ can be rewritten as
$$\sum_{\substack{p(\Pi_1)=\rho_1 \\ p(\Pi_2)=\rho_2} }\!\!(-1)^{t-1} \,\left((t-1)! -kl\,(t-2)!+\binom{k}{2}\binom{l}{2}\,2!\,(t-3)!-\binom{k}{3}\binom{l}{3}\,3!\,(t-4)!+\cdots\right),$$
where $t=k+l$. However, for every possible value of $k \geq 1$ and $l \geq 1$, the term in parentheses vanishes. Indeed, assuming for instance $k \leq l$, we look at
\begin{align*}
&(k+l-1)!\,\sum_{x=0}^{k} (-1)^x \, \binom{k}{x}\,\binom{l}{x} \,\binom{k+l-1}{x}^{\!\!-1}\\
 &=k!\,(l-1)!\,\sum_{x=0}^{k} (-1)^x \,\binom{l}{x} \,\binom{k+l-1-x}{k-x}\\
&=k!\,(l-1)!\,\binom{k-1}{k}=0
\end{align*}
by using Riordan's array rule for the second identity.  
\end{proof}\bigskip

Thus, $F$ vanishes on pairings $\rho$ associated to labelled trees with more than one edge of height $1$. In other words, if $F(\rho)\neq 0$, then $\{1,2r\}$ is a pair in $\rho$, and we can look at the restricted pairing $\tilde\rho=\rho_{|\lle 2,2r-1\rre}$, which is of size $2r-2$; and we can consider $F$ as a functional on $\pym_{2r-2}$. 
To avoid any ambiguity, we denote this new functional
$$G(\rho \in \pym_{2r})=\sum_{p(\Pi)=\rho}(-1)^{\ell(\Pi)}\, (\ell(\Pi))!$$
We then expect the formula $G(\rho)=(-1)^{r}\,\prod_{e \in E(T(\rho))} h(e)$. We proceed by induction on labelled rooted planar trees, and we look at the action of adding a leave of label $1$ to the tree, and of increasing a label of an edge by $1$. To fix the ideas, it is convenient to consider the following example of pairing $\rho$, and the associated set of set partitions $\Pi$ with $p(\Pi)=\rho$. The pairing $\rho$ of Figure \ref{fig:setpartition} is associated to the labelled planar rooted tree on Figure \ref{fig:labelledtree}, and it has functional $G(\rho)=(-1)^3\,3! + 2\times (-1)^2\,2!=-2$. We denote $N(l,\rho)$ the number of set partitions such that $p(\Pi)=\rho$ and $\ell(\Pi)=l$. Hence, $$G(\rho)=\sum_{l=1}^{r}N(l,\rho)\,(-1)^l\,l!$$
\clearpage

\figcap{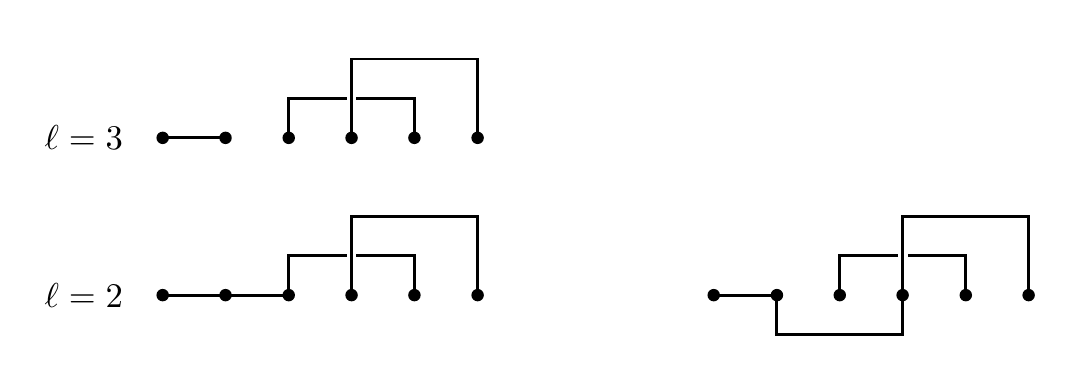}{A pairing  of size $2r=6$ (the upper diagram) and the associated set of set partitions, which contains 3 elements.\label{fig:setpartition}}\vspace{-10mm}
\figcap{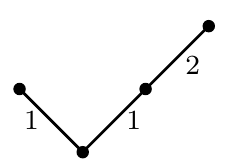}{The labelled planar rooted tree associated to the pairing of Figure \ref{fig:setpartition}.\label{fig:labelledtree}}

\begin{enumerate}
 \item \emph{Adding an edge}. Suppose that one adds an edge with label $1$, to obtain for instance: 
\figcap{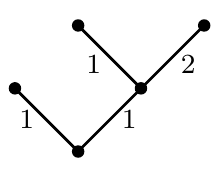}{Addition of an new edge of label 1 to the planar rooted tree.}
Set $\rho'$ for the new pairing; notice that it is obtained from $\rho$ by inserting a simple bond \includegraphics{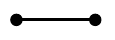}. The set partitions $\Pi'$ with $p(\Pi')=\rho'$ are of two kinds:\vspace{1mm}
\begin{enumerate}
\item those where the new bond is left alone. They all come from a set partition $\Pi$ with $p(\Pi)=\rho$ by simply inserting the new bond:
\end{enumerate}
\end{enumerate}
\figcap{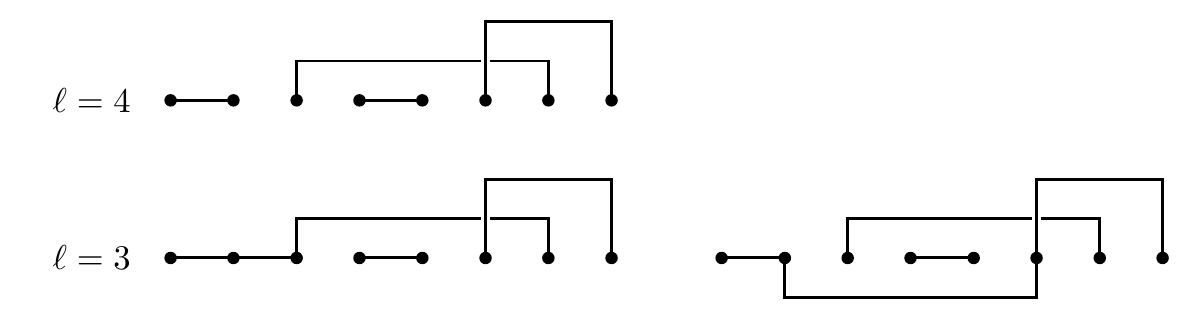}{Set partitions where the new bond is left alone.}
\begin{enumerate}
\item[]
\begin{enumerate}
\item[] These terms give the following contribution to $G(\rho')$:
$$G_{(\textrm{a})}(\rho')=-\sum_{l=1}^{r}N(l,\rho)\,(-1)^{l}\,(l+1)!.$$
\vspace{1mm}
\setcounter{enumii}{1}
\item those where the new bond is linked to another part of a set partition $\Pi$ with $p(\Pi)=\rho$. Starting from a set partition $\Pi$ with $p(\Pi)=\rho$, the number of parts of $\Pi$ that can actually receive the new bond is $\ell(\Pi)-(h(e)-1)$, because the new bond cannot be linked to the $h(e)-1$ parts that go above him. In our example:
\vspace{1mm}
\end{enumerate}
\end{enumerate}
\figcap{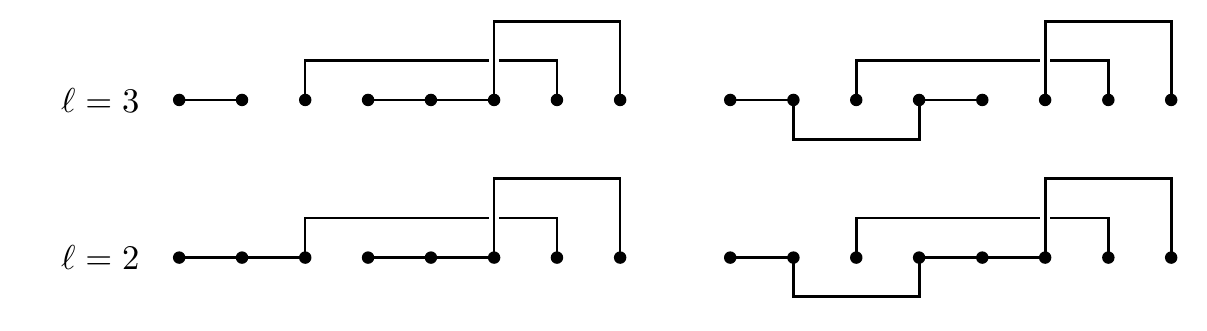}{Set partitions where the new bound is integrated in another part.}

\begin{enumerate}
\item[]~ \setcounter{enumi}{1}
\begin{enumerate}
\item[] These other terms give the following contribution to $G(\rho')$:
$$G_{(\textrm{b})}(\rho')=\sum_{l=1}^{r}N(l,\rho)\,(-1)^{l}\,l!\,(l+1-h(e)).$$
\vspace{1mm}
\end{enumerate}
We conclude that $G(\rho')=G_{(\textrm{a})}(\rho')+G_{(\textrm{b})}(\rho')=-h(e)\,G(\rho)$, so the formula for $G$ stays true when one adds an edge of label $1$.
\vspace{3mm}

\item \emph{Raising a label}. As explained before, raising a label corresponds to adding a simple crossing to the pairing $\rho$, which is done by exchanging two ends $b$ and $d$ of two simply nested pairs $\{a<b\}$ and $\{c<d\}$ of $\rho$. This does not change the structure of the set of even set partitions $\Pi$ with $p(\Pi)=\rho$; that is, $N(l,\rho)=N(l,\rho')$ for every $l$. So, the formula for $G$ also stays true when one raises a label.
 \vspace{3mm}
 \end{enumerate}
Since every labelled rooted tree is obtained inductively from the empty tree by adding edges and raising labels, the proof of Proposition \ref{prop:functionalF} is done.

\subsubsection{Expansion of the joint cumulants as sums over Dyck paths}
Recall that $x^{\nu}$ stands for $x^{(\mathbf{a}_2-\mathbf{a}_1)+\cdots+(\mathbf{a}_{2r}-\mathbf{a}_{2r-1})}$ if $\nu$ is the pairing $\{a_1<a_2\},\ldots,\{a_{2r-1}<a_{2r}\}$. We adopt the same notations with Dyck paths and planar rooted trees, so $x^\delta$ or $x^T$ stands for $x^\nu$ if $\delta=\delta(\nu)$ or if $T=T(\nu)$. We also denote $\dym_{2r}^*$ the image of $\dym_{2r-2}$ in $\dym_{2r}$ by the operation $\delta \mapsto \delta^+$. Notice that if $\Delta=(\delta(T))^+$ with $T$ tree with $r-1$ edges, then
$$\prod_{e \in E(T)} h(e)\,(h(e)+1)= \prod_{i=1}^{2r-1} \Delta_i,$$
$\Delta_i$ denoting the value of the Dyck path $\Delta$ after $i$ steps. Starting from Equation \eqref{eq:jointcumulant2} and using the explicit formulas that we have obtained for $N(\nu)$ and $F(\nu)$, we therefore get:

\begin{theorem}
\label{thm:jointcumulant}
For every indices $\mathbf{1}\leq\cdots\leq\mathbf{2r}$, 
$$\kappa(\sigma(\mathbf{1}),\ldots,\sigma(\mathbf{2r}))=(-1)^{r-1}\,\sum_{\delta \in \dym_{2r}^*} \left(\prod_{i=1}^{2r-1}\delta_i\right)\, x^{\delta}.$$    
\end{theorem}

\begin{example}
The two non-crossing pairings of size $4$ are \includegraphics{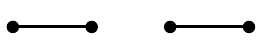} and \includegraphics{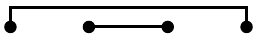}, the associated powers of $x$ are 
$$x^{(\mathbf{6}-\mathbf{1})+(\mathbf{5}-\mathbf{4})+(\mathbf{3}-\mathbf{2})} \quad\text{and}\quad x^{(\mathbf{6}-\mathbf{1})+(\mathbf{5}-\mathbf{2})+(\mathbf{4}-\mathbf{3})},$$
and the associated quantities $G(\nu)$ are $4$ and $12$, so, with $r=3$,
$$\kappa(\sigma(\mathbf{1}),\ldots,\sigma(\mathbf{6}))=4\,x^{\mathbf{6}+\mathbf{5}+\mathbf{3}-\mathbf{4}-\mathbf{2}-\mathbf{1}} + 12\, x^{\mathbf{6}+\mathbf{5}+\mathbf{4}-\mathbf{3}-\mathbf{2}-\mathbf{1}}.$$    
\end{example}
\bigskip

Theorem \ref{thm:jointcumulant} has several easy corollaries. First of all, we see immediately from it that the sign of a joint cumulant of spins is prescribed, which was \emph{a priori} non-obvious. On the other hand, applying Theorem \ref{thm:jointcumulant} to the case $r=1$ yields $$\kappa(\sigma(i),\sigma(j))=x^{|j-i|},$$ that is, the correlation between two spins decreases exponentially with the distance between the spins. More generally, one can use Theorem \ref{thm:jointcumulant} to get a useful bound on cumulants. Notice that the minimal exponent of $x$ that appears in the right-hand side of the formula is $$x^{\mathbf{(2r)}+(\mathbf{(2r-1)}-\mathbf{(2r-2)})+(\mathbf{(2r-3)}-\mathbf{(2r-4)})+\cdots+(\mathbf{3}-\mathbf{2})-\mathbf{1}}.$$
Indeed, it is easily seen that the exponent of $x$ in $x^T$ increases when one makes a rotation of a leaf of $T$ in the sense of Tamari (\emph{cf.} \cite{Tam62}). Since all trees are generated by leaf rotations from the tree with all edges of height $1$ (\emph{cf.} \cite{Knu04}), the previous claim is shown. It follows that 
$$\left|\kappa(\sigma(\mathbf{1}),\ldots,\sigma(\mathbf{2r})) \right|\leq \left(\sum_{\delta \in \dym_{2r}^*} \prod_{i=1}^{2r-1} \delta_i \right) x^{\mathbf{(2r)}+(\mathbf{(2r-1)}-\mathbf{(2r-2)})+\cdots+(\mathbf{3}-\mathbf{2})-\mathbf{1}}.$$
The quantity $$Q(r)=\sum_{\delta \in \dym_{2r}^*} \left(\prod_{i=1}^{2r-1} \delta_i\right) = \sum_{T \in \tym_{r-1}} \left(\prod_{e \in E(T)} h(e)\,(h(e)+1)\right)$$ has for first values $1,2,16,272,7936,\ldots,$ and a simple bound on $Q(r)$ is $(2r-2)!$, see Proposition \ref{prop:polynomial} hereafter. Hence, a generalization of the exponential decay of covariances is given by:

\begin{proposition}
For any positions of spins $i_1\leq i_2 \leq \cdots \leq i_{2r}$, 
$$\left|\kappa(\sigma(i_1),\ldots,\sigma(i_{2r})) \right|\leq (2r-2)!\,\, x^{i_{2r}+(i_{2r-1}-i_{2r-2})+\cdots+(i_{3}-i_{2})-i_{1}}.$$
\end{proposition}
\medskip

\subsection{Bounds on the cumulants of the magnetization}
As explained in the introduction, we now have to gather the estimates given by Theorem \ref{thm:jointcumulant} to get the asymptotics of the cumulants $\kappa^{(2r)}(M_n)$ of the magnetization. 

\subsubsection{Reordering of indices and compositions} Since the joint cumulants of spins have been computed for ordered spins $i_1\leq i_2\leq\ldots\leq i_{2r}$, in the right-hand side of the expansion
 $$\kappa^{(2r)}(M_n)=\sum_{i_1,\ldots,i_{2r}=1}^n\kappa(\sigma(i_1),\ldots,\sigma(i_{2r})),$$
we need to reorder the indices $i_1,\ldots,i_{2r}$, and take care of the possible identities between these indices. We shall say that a sequence of indices $i_1,\ldots,i_{r}$ has type $c=(c_1,\ldots,c_l)$ with the $c_i$ positive integers and $|c|=\sum_{i=1}^l c_i = r$ if, after reordering, the sequence of indices writes as
$$i_1'=i_2'=\ldots=i_{c_1}'<i_{c_1+1}'=i_{c_1+2}'=\cdots = i_{c_1+c_2}'<i_{c_1+c_2+1}' = \cdots.$$
Here, $i_k'$ stands for the $k$-th element of the reordered sequence. For instance, the sequence of indices $(3,2,3,5,1,2)$ becomes after reordering $(1,2,2,3,3,5)$, so it has type $(1,2,2,1)$. The type of a sequence of indices of length $r$ can be any composition of size $r$, and we denote $\cym_{r}$ the set of these compositions. Conversely, given a composition of size $r$ and length $l$, in order to construct a sequence of indices $(i_1,\ldots,i_r)$ with type $c$ and with values in $\lle 1,n\rre$, one needs:\vspace{1mm}
\begin{itemize}
\item  to choose which indices $i$ will fall into each class $(i_1',\ldots,i_{c_1}')$, $(i_2',\ldots,i_{c_1+c_2}')$, \emph{etc}.; there are
$$\binom{r}{c}=\frac{r!}{c_1!\,c_2!\cdots c_l!}$$
possibilities there.
\vspace{1mm}
\item and then to choose $1\leq j_1<j_2<\cdots <j_l \leq n$ so that $j_1=i_1'=\cdots=i_{c_1}'$, $j_2=i_2'=\cdots=i_{c_1+c_2}'$, \emph{etc}.\vspace{1mm}
\end{itemize}
As a consequence,
\begin{align*}\kappa^{(2r)}(M_n)&=\sum_{c \in \cym_{2r}} \sum_{1\leq j_1<j_2<\cdots <j_{\ell(c)} \leq n} \binom{2r}{c}\, \kappa\left(\sigma(j_1)^{c_1},\ldots,\sigma(j_{\ell(c)})^{c_{\ell(c)}}\right)\\
&=(-1)^{r-1}\sum_{c\in \cym_{2r}}\sum_{\delta \in \dym_{2r}^*}  \binom{2r}{c}\, C(\delta)\,B(n,c,\delta)
\end{align*}
where $C(\delta)=\prod_{i=1}^{2r-1} \delta_i$ is the quantity computed in the previous paragraph, and 
$$B(n,c,\delta)=\sum_{1\leq j_1<j_2<\cdots <j_{\ell(c)} \leq n} x^{\sum_{\{a<b\}\in \nu(\delta)} (i_b-i_a)},$$
the indices $i$ being computed from the indices $j$ according to the rule previously explained, namely,
\begin{align*}
j_1&=i_1=\cdots=i_{c_1};\\
j_2&=i_{c_1+1}=\cdots=i_{c_1+c_2};\\
\vdots & \qquad\qquad \qquad \vdots\\
j_{\ell(c)}&=i_{c_1+\cdots+c_{\ell(c)-1}+1}=\cdots=i_{2r}.
\end{align*}

\begin{example}
Suppose $r=1$. There are two compositions of size $2$, namely, $(2)$ and $(1,1)$, and one trivial tree with $0$ edge; therefore,
\begin{align*}
\kappa^{(2)}(M_n)&=B(n,(2),\bullet)+2\,B(n,(1,1),\bullet)\\ 
&=\sum_{ j_1 =1}^{ n} 1 \,\,\,\,+\,\,2\!\!\!\sum_{1\leq j_1<j_2 \leq n} x^{j_2-j_1}. \end{align*}
The double geometric sum has the same asymptotics as $\sum_{j_1 =1}^n\sum_{j_2=j_1+1}^\infty x^{j_2-j_1} = n\,\frac{x}{1-x},$
so $$\kappa^{(2)}(M_n) \simeq n\,\frac{1+x}{1-x} = n\,\E^{2\beta}.$$
\end{example}\bigskip

It is not hard to convince oneself that the approximation performed in the previous example can be done in any case, so that a correct estimate of $B(n,c,\delta)$ is $n\, B(c,\delta)$, with
$$B(c,\delta)=\sum_{0=j_1<j_2<\cdots<j_{\ell(c)}} x^{\sum_{\{a<b\}\in \nu(\delta)} (i_b-i_a)}.$$
In this new expression, the indices $j$ are unbounded (except the first one, fixed to $0$), and what we mean by approximation is that $$n\,B(c,\delta)= B(n,c,\delta)+O(1),$$ with a positive remainder corresponding to terms of the geometric series with indices larger than $n$. So:

\begin{proposition}\label{prop:firstestimate}
An upper bound, and in fact an estimate of $|\kappa^{(2r)}(M_n)|$ is
$$|\kappa^{(2r)}(M_n)|\leq n\,\sum_{c \in\cym_{2r}}\sum_{\delta \in \dym_{2r}^*} \binom{2r}{c}\,B(c,\delta)\,C(\delta).$$
\end{proposition}

\subsubsection{Computation of the functional $B$}
There is a simple algorithm that allows to compute $B(c,\delta)$ for any Dyck path $\delta$ and any composition $c$. Let us explain it with the path $\delta$ associated to the non-crossing pairing $\nu$ of Figure \ref{fig:restriction} and with the composition $c=(3,2,1,2,2)$. This composition $c$ corresponds to some identifications of indices, which we make appear on the diagram of the pairing $\nu$ as follows:
\figcap{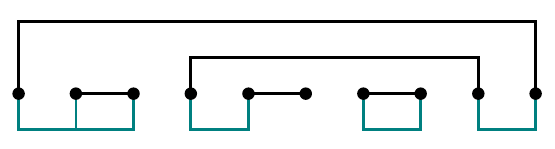}{Identifications of indices corresponding to the composition $c=(3,2,1,2,2)$.}
We now contract the green edges added above, obtaining thus: \vspace{-2mm}
\figcap{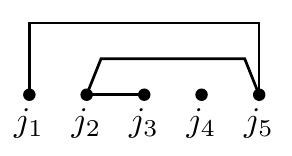}{Contraction of the diagram of a non-crossing partition along a composition.\label{fig:contracteddiagram}}
This new diagram corresponds to the following simplification of the sum $B(c,\delta)$:
\begin{align*}
B(c,\delta)&= \sum_{0=i_1=i_2=i_3<i_4=i_5<i_6<i_7=i_8<i_9=i_{10}} x^{i_{10}+i_9+i_8-i_7+i_6-i_5-i_4+i_3-i_2-i_1}\\
&=\sum_{0=i_1<i_4<i_6<i_7<i_9} x^{2i_9+i_5-2i_4-i_1}\qquad \text{because of the identities of indices};\\
&=\sum_{0=j_1<j_2<j_3<j_4<j_5} x^{(j_5-j_1)+(j_5-j_2)+(j_3-j_2)}\qquad \,\,\text{by relabeling the indices}.
\end{align*}
So, the new diagram, which we call the \emph{contraction of $\nu$ along $c$} and denote $\nu\!\downarrow_c$, can be read similarly as the previous diagrams of pairings, that is to say that
$$B(c,\delta)=\sum_{0=j_1<j_2<j_3<j_4<j_5} x^{(\nu(\delta))\downarrow_c},$$
where $x^{\nu\downarrow_c}$ stands for the product of factors $x^{b-a}$, $\{a<b\}$ running over the bonds of the contracted diagram $\nu\!\downarrow_c$.\bigskip

Given a contracted diagram $\rho=\nu\!\downarrow_c$ of length $\ell(c)$, denote $\delta_1(\rho)$ the number of bonds opened between $j_1$ and $j_2$; $\delta_2(\rho)$ the number of bonds opened between $j_2$ and $j_3$; $\delta_3(\rho)$ the number of bonds opened between $j_3$ and $j_4$; \emph{etc.} up to $\delta_{\ell(c)-1}(\rho)$. For instance, in the previous example, there is one bond opened between $j_1$ and $j_2$ (the one starting from $j_1$); $3$ bonds opened between $j_2$ and $j_3$ (the previous bond, which has not been closed, and the two bonds starting from $j_2$); and $2$ bonds opened between $j_3$ and $j_4$ and between $j_4$ and $j_5$. So $(\delta_1,\delta_2,\delta_3,\delta_4)=(1,3,2,2)$.

\begin{proposition}\label{prop:functionalB}
Set $\rho=(\nu(\delta))\!\downarrow_c$. One has
$$B(c,\delta)=\prod_{i=1}^{\ell(c)-1}\frac{x^{\delta_i(\rho)}}{1-x^{\delta_i(\rho)}}.$$
\end{proposition}

\begin{example}
Consider the previous contracted diagram $\rho_5$, and the corresponding sum $$B_5=\sum_{0=j_1<j_2<j_3<j_4<j_5} x^{(j_5-j_1)+(j_5-j_2)+(j_3-j_2)}.$$ We reduce inductively the size of the contracted diagram as follows. We first write
\begin{align*}B_5&=\sum_{0=j_1<j_2<j_3<j_4<j_5} x^{2(j_5-j_4)+(j_4-j_1)+(j_4-j_2)+(j_3-j_2)}\\
&=\left(\sum_{0=j_1<j_2<j_3<j_4} x^{(j_4-j_1)+(j_4-j_2)+(j_3-j_2)}\right)\left( \sum_{j_5=j_4+1}^\infty x^{2(j_5-j_4)} \right)\\
&=\frac{x^2}{1-x^2}\left(\sum_{0=j_1<j_2<j_3<j_4} x^{(j_4-j_1)+(j_4-j_2)+(j_3-j_2)}\right)=\frac{x^{\delta_4}}{1-x^{\delta_4}}\,B_4,
\end{align*}
where $B_4$ is the sum corresponding to the diagram $\rho_4$ which is obtained from $\rho_5$ by identifying $j_4$ and $j_5$:
\figcap{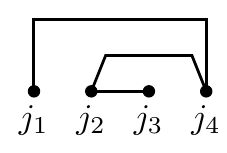}{Reduction of the diagram of Figure \ref{fig:contracteddiagram}.}
We can then do it again to go to size $3$:
\begin{align*}B_4&=\sum_{0=j_1<j_2<j_3<j_4} x^{2(j_4-j_3)+(j_3-j_1)+(j_3-j_2)+(j_3-j_2)}\\
&=\left(\sum_{0=j_1<j_2<j_3} x^{(j_3-j_1)+2(j_3-j_2)}\right)\left( \sum_{j_4=j_3+1}^\infty x^{2(j_4-j_3)} \right)\\
&=\frac{x^2}{1-x^2}\left(\sum_{0=j_1<j_2<j_3} x^{(j_3-j_1)+2(j_3-j_2)}\right)=\frac{x^{\delta_3}}{1-x^{\delta_3}}\,B_3,
\end{align*}
where $B_3$ is the sum corresponding to the diagram $\rho_3$ which is obtained from $\rho_4$ by identifying $j_3$ and $j_4$:
\figcap{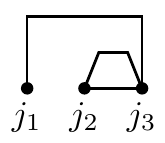}{Further reduction of the diagram of Figure \ref{fig:contracteddiagram}.}
Two more operations yield similarly the factors $\frac{x^{\delta_2}}{1-x^{\delta_2}}$ and $\frac{x^{\delta_1}}{1-x^{\delta_1}}$.
\end{example}

\begin{proof}[Proof of Proposition \ref{prop:functionalB}]
The algorithm presented above on the example gives clearly a proof of the formula by induction on $\ell(c)$. Indeed, at each step of the induction, the term that is factorized is $$\sum_{j_{\ell(c)}=j_{\ell(c)-1}+1}^\infty x^{\delta_{\ell(c)-1}(j_{\ell(c)}-j_{\ell(c)-1})}=\frac{x^{\delta_{\ell(c)-1}}}{1-x^{\delta_{\ell(c)-1}}},$$
because $\delta_{\ell(c)-1}$ is the number of bonds ending at $j_{\ell(c)}$. Then, as for the other factor, one obtains it by replacing $j_{\ell(c)}$ by $j_{\ell(c)-1}$ in the sum $B(c,\delta)$, and this amounts to do the identification between $j_{\ell(c)-1}$ and $j_{\ell(c)}$ in the contracted diagram. This identification and reduction to lower length does not change the values $\delta_1,\ldots,\delta_{\ell(c)-2}$, so the formula is proven.
\end{proof}
\bigskip

We recall that a descent of a composition $c=(c_1,\ldots,c_\ell)$ is one of the integers $$c_1, c_1+c_2, c_1+c_2+c_3,\ldots,c_1+\cdots+c_{\ell-1}.$$ For instance, the descents of $c=(3,2,1,2,2)$ are $3$, $5$, $6$ and $8$. The set of descents $D(c)$ of a composition $c$ of size $r$ can be any subset of $\lle 1,r-1\rre$, so in particular, $\mathrm{card}\,\cym_r=2^{r-1}$. The contraction of diagrams along compositions presented at the beginning of this paragraph satisfies the rule:
$$\{\delta_1(\rho),\ldots,\delta_{\ell(c)-1}(\rho)\}=\{\delta_d,\,\,d \in D(c)\} \quad \text{if }\rho=(\nu(\delta))\!\downarrow_c.$$
So, $B(c,\delta)=\prod_{d \in D(c)} \frac{x^{\delta_d}}{1-x^{\delta_d}}$, and Proposition \ref{prop:firstestimate} becomes:

\begin{theorem}
An upper bound, and  in fact an estimate of $|\kappa^{(2r)}(M_n)|$ is
$$\frac{|\kappa^{(2r)}(M_n)|}{n}\leq\sum_{c \in \cym_{2r}} \sum_{\delta \in \dym_{2r}^*} A(c)\,B(c,\delta)\,C(\delta)$$
with $A(c)=\binom{2r}{c}$, $B(c,\delta)=\prod_{d \in D(c)}\frac{x^{\delta_d}}{1-x^{\delta_d}}$ and $C(\delta)=\prod_{i=1}^{2r-1}\delta_i$.
\end{theorem}\bigskip

\begin{example}
Suppose $r=2$. There is one Dyck path in $\dym_4^*$, with $C(\delta)=2$ since $\delta_1=\delta_3=1$ and $\delta_2=2$. The compositions of size $4$ are $(4)$, $(3,1)$, $(2,2)$, $(1,3)$, $(2,1,1)$, $(1,2,1)$, $(1,1,2)$ and $(1,1,1,1)$; their contributions $A(c)\,B(c,\delta)$ are equal to
$$1,\,\,\frac{4x}{1-x},\,\,\frac{6x^2}{1-x^2},\,\,\frac{4x}{1-x},\,\,\frac{12x^3}{(1-x)(1-x^2)},\,\,\frac{12x^2}{(1-x)^2},\,\,\frac{12x^3}{(1-x)(1-x^2)},\,\,\frac{24x^4}{(1-x)^2(1-x^2)}.$$
So,
\begin{align*}|\kappa^{(4)}(M_n)| &\simeq 2n\,\left(1+\frac{8x}{1-x}+\frac{6x^2}{1-x^2}+\frac{12x^2}{(1-x)^2}+\frac{24x^3}{(1-x)(1-x^2)}+\frac{24x^4}{(1-x)^2(1-x^2)}\right)\\
&\simeq 2n \,\,\frac{(1+x)(1+4x+x^2)}{(1-x)^3}=n\,(3\,\E^{6\beta}-\E^{2\beta}).
\end{align*}
\end{example}

\subsubsection{Explicit bound on cumulants and the mod-Gaussian convergence}
By examining the asymptotics of the first cumulants written as rational functions in $x$, one is lead to the following result. Set 
$$P_r(x) = \left(\sum_{c \in \cym_{2r}} \sum_{\delta \in \dym_{2r}^*} A(c)\,B(c,\delta)\,C(\delta)\right)\,(1-x)^{2r-1}.$$
For instance, $P_1(x)=1+x$ and $P_2(x)=2\,(1+x)(1+4x+x^2)$. 

\begin{proposition}\label{prop:polynomial}
For every $r \geq 1$ and every $x \in (0,1)$, $$0 \leq P_r(x) \leq \frac{(2r)!}{r!}\,\frac{(2r-2)!}{(r-1)!}.$$
\end{proposition}

\begin{proof}For every composition $c$ and every path $\delta$, $B(c,\delta)\,(1-x)^{2r-1}$ is a non-negative and convex function of $x$ on $[0,1]$. Therefore, $0\leq P_r(x) \leq x\,P_r(0)+(1-x)\,P_r(1)$. When $x=1$, all the rational functions $B(c,\delta)\,(1-x)^{2r-1}$ vanish, except when $c$ has $2r-1$ descents, that is to say that $c=(1,1,\ldots,1)$. Then, $A(c)=(2r)!$, and
$$\lim_{x \to 1}\, B(c,\delta)= \prod_{i=1}^{2r-1} \frac{1}{\delta_i}=\frac{1}{C(\delta)}.$$
Therefore, 
$$P_r(1)=(2r)!\,(\card \,\dym_{2r}^*) = \frac{(2r)!}{r!}\,\frac{(2r-2)!}{(r-1)!}.$$
On the other hand, when $x=0$, all the rational functions $B(c,\delta)\,(1-x)^{2r-1}$ vanish, except when $c$ has no descent, that is to say that $c=(2r)$. Then, $A(c)=1$ and 
$$P_r(0)=Q(r)=\sum_{\delta \in \dym_{2r}^*}A(\delta).$$
Among all Dyck paths in $\dym_{2r}^*$, the one with the maximal product of values $G(\delta)$ is 
$(0,1,2,\ldots,r-1,r,r-1,\ldots,2,1,0)$. So,
$$
P_r(0) \leq r!\,(r-1)!\,(\card \,\dym_{2r}^*) = (2r-2 )! \leq P_r(1).
$$
It follows that $P_r(x) \leq x\,P_r(1)+(1-x)\,P_r(1)=P_r(1)$. 
\end{proof}
\bigskip

\begin{corollary}
For every $r$,
$$|\kappa^{(2r)}(M_n)|\leq n\,(2r-1)!!\,(2r-3)!! \,(\E^{2\beta}+1)^{2r-1}.$$
\end{corollary}

\begin{proof}
Indeed,
\begin{align*}
|\kappa^{(2r)}(M_n)| &\leq n \, \left(\sum_{c \in \cym_{2r}} \sum_{\delta \in \dym_{2r}^*} A(c)\,B(c,\delta)\,C(\delta)\right) = n\,\frac{P_r(x)}{(1-x)^{2r-1}} \\
&\leq n\,\frac{P_r(1)}{(1-x)^{2r-1}} = n\,\left(\frac{1}{1-x}\right)^{2r-1}\,\frac{(2r)!}{r!}\,\frac{(2r-2)!}{(r-1)!}. 
\end{align*}
Replacing $x$ by $\tanh \beta$ allows to conclude, and this gives another proof of Theorem \ref{thm:modgaussising}. We rewrite the logarithm of the Laplace transform of $n^{-1/4}M_n$ as
$$\sum_{r=1}^\infty \frac{\kappa^{(2r)}(M_n)}{(2r)!}\,z^{2r}\,n^{-r/2} = \frac{\kappa^{(2)}(M_n)\,z^2}{2n^{1/2}}+\frac{\kappa^{(4)}(M_n)\,z^4}{24n}+\sum_{r=3}^\infty \frac{\kappa^{(2r)}(M_n)}{(2r)!}\,z^{2r}\,n^{-r/2}.$$
The series on the right-hand side is smaller than
\begin{align*}\sum_{r=3}^\infty \frac{(2r-1)!!\,(2r-3)!!}{(2r)!} \, (\E^\beta+1)^{2r-1}\,z^{2r}\,n^{1-r/2} &\leq n^{-1/2}\sum_{r=3}^\infty ((\E^{2\beta} + 1)z)^{2r} \,n^{-(r-3)/2}\\
&\leq n^{-1/2}\,\,\frac{((\E^{2\beta} + 1)z)^{6}}{1-((\E^{2\beta} + 1)z)^{2}\,n^{-1/2}},
\end{align*}
so it goes uniformly to zero on every compact set of the plane. On the other hand, we have seen that $\kappa^{(2)}(M_n)=n\,\E^{2\beta} - O(1)$ and $-\kappa^{(4)}(M_n)=n\,(3\,\E^{6\beta}-\E^{2\beta}) - O(1)$, so we conclude that
$$\esper\!\left[\E^{z\,\frac{M_n}{n^{1/4}}}\right]\,\E^{-\frac{n^{1/2}\E^{2\beta}\,z^2}{2}}=\E^{-\frac{(3\,\E^{6\beta}-\E^{2\beta})\,z^4}{24}}\,\left(1+O(n^{-1/2})\right),$$
and this is indeed the content of Theorem \ref{thm:modgaussising}.
\end{proof}

\bibliographystyle{alpha}
\bibliography{ising}

\end{document}